\documentclass{amsart}
\usepackage{amssymb}
\usepackage[arrow,matrix]{xy}
\usepackage{graphicx}
\usepackage{color}
\usepackage{marginnote}

\theoremstyle{plain}
\newtheorem{theorem}{Theorem}[section]
\newtheorem{proposition}[theorem]{Proposition}
\newtheorem{lemma}[theorem]{Lemma}
\newtheorem{corollary}[theorem]{Corollary}

\theoremstyle{definition}
\newtheorem*{acknowledgements}{Acknowledgements}

\newtheorem{conjecture}[theorem]{Conjecture}
\newtheorem{definition}[theorem]{Definition}

\theoremstyle{remark}
\newtheorem{remark}[theorem]{Remark}

\newtheorem{notation}[theorem]{Notation}

\renewcommand{\bar}{\overline}
\renewcommand{\tilde}{\widetilde}
\renewcommand{\vec}{\overrightarrow}
\renewcommand{\hat}{\widehat}

\newcommand{\C}{\mathbb{C}}

\newcommand{\N}{\mathbb{N}}

\newcommand{\R}{\mathbb{R}}
\newcommand{\U}{\mathcal{U}}
\newcommand{\Z}{\mathbb{Z}}

\newcommand{\modz}{~(\operatorname{mod}\Z)}

\title[On the Chern numbers for pseudo-free circle actions]{On the Chern numbers for pseudo-free circle actions}

\author[B. H. An]{Byung Hee An}
\address{Center for Geometry and Physics, Institute for Basic Science (IBS), Pohang, Republic of Korea 37673}
\email{anbyhee@ibs.re.kr}

\author[Y. Cho]{Yunhyung Cho}
\address{Department of Mathematics Education, Sungkyunkwan University, 25-2, Sungkyunkwan-Ro,
Jongno-Gu, Seoul 03063, Republic of Korea}
\email{yunhyung@skku.edu}


\begin{document}

\date{March 5, 2017}

\begin{abstract}	
Let $(M,\psi)$ be a $(2n+1)$-dimensional oriented closed manifold with a pseudo-free $S^1$-action 
$\psi : S^1 \times M \rightarrow M$. 
We first define a \textit{local data} $\mathcal{L}(M,\psi)$ of the action $\psi$ which consists of 
pairs $(C, (p(C) ; \overrightarrow{q}(C)))$ where $C$ is an exceptional orbit, $p(C)$ is the order of isotropy subgroup of $C$, and $\overrightarrow{q}(C) \in (\Z_{p(C)}^{\times})^n$ is a vector whose entries are the weights of the slice representation of $C$. 
In this paper, we give an explicit formula of the Chern number $\langle c_1(E)^n, [M/S^1] \rangle$ modulo $\Z$ in terms of 
the local data, where $E  = M \times_{S^1} \C$ is the associated complex line orbibundle 
over $M/S^1$.
Also, we illustrate several applications to various problems arising in equivariant symplectic topology. 
\end{abstract}
\maketitle

\section{Introduction}

Let $N$ be a $2n$-dimensional oriented closed manifold and $S^1 = \{ z \in \C ~|~ |z| = 1 \}$ be the unit circle group. 
Suppose that there is an effective $S^1$-action $\phi : S^1 \times N \rightarrow N$ on $N$.
The localization theorem due to Atiyah-Bott \cite{AB} and Berline-Vergne \cite{BV} is a very powerful technique for computing global (topological) invariants of $N$ in terms of local data 
\[
\mathcal{L}(N, \phi) = \{(F, \nu_{S^1}(F))\}_{F \subset N^{S^1}}
\]
where $F$ is a connected component of the fixed point set $N^{S^1}$ and $\nu_{S^1}(F)$ is an $S^1$-equivariant normal bundle of $F$ in $N$. In particular if $N$ admits an $S^1$-invariant almost complex structure, then we can compute the Chern numbers of the tangent bundle $TN$ in terms of the local data $\mathcal{L}$. 

In this paper, we attempt to find an odd dimensional analogue of the ABBV-localization theorem in the sense that if we have a $(2n+1)$-dimensional oriented closed manifold $M$ equipped with an effective fixed-point-free $S^1$-action $\psi : S^1 \times M \rightarrow M$, then our aim is to find a method for computing global invariants in terms of local data. Here, local data means 
\[
\mathcal{L}(M, \psi) = \{ (M^{\Z_p}, \nu_{S^1}(M^{\Z_p}))\}_{p \in \N, p>1}
\]
where $\Z_p$ is the cyclic subgroup of $S^1$ of order $p$, $M^{\Z_p}$ is a submanifold of $M$ fixed by $\Z_p $, and $ \nu_{S^1}(M^{\Z_p})$ is an $S^1$-equivariant normal bundle of $M^{\Z_p}$ in $M$.
To do this, let us consider the following commutative diagram 
\[
\xymatrixcolsep{5pc}\xymatrix{
M \times \C \ar[r]_{q}^{/S^1} \ar[d]_{\pi} & M \times_{S^1} \C  = E\ar[d]^{\pi}\\
M \ar[r]_{q}^{/S^1} &  M / S^1 = B}
\]
where $S^1$ acts on $M \times \C$ by $$t \cdot (x,z) = (t\cdot x, tz)$$ for every 
$t \in S^1$ and $(x,z) \in M \times \C$. 

If the action is free, then $B$ is a smooth manifold and $E$ becomes a complex line bundle over $B$ with the first Chern class $c_1(E) \in H^2(B;\Z)$. In particular, the Chern number $\langle c_1(E)^n, [B] \rangle$ is an integer where $[B] \in H_{2n}(B;\Z)$ is the fundamental homology class of $B$.  

On the other hand, if the action is not free, then $B$ is an orbifold with cyclic quotient singularities and $E$ becomes a complex line orbibundle over $B$. Then the first Chern class $c_1(E) \in H^2(B,\R)$ is defined, via the Chern-Weil construction, as a cohomology class represented by a differential 2-form $\Theta_{\alpha}$ on $B$ where $\alpha$ is a normalized connection 1-form on $M$ and $\Theta_{\alpha}$ is a 2-form on $B$ such that $d\alpha = q^* \Theta_{\alpha}$. Then the Chern number of $E$ is given by
\[
\langle c_1(E)^n, [B] \rangle = \int_B \Theta_{\alpha} \wedge \Theta_{\alpha} \wedge \cdots \wedge \Theta_{\alpha} = \int_M \alpha \wedge (d\alpha)^n
\]
which is a rational number in general (see \cite[Theorem 1]{W}). 
However, the local data $\mathcal{L}(M, \psi)$ does not detect any information about free orbits by definition of $\mathcal{L}(M, \psi)$. 
In fact, if the $S^1$-action is free, then the local data $\mathcal{L}(M, \psi)$ is an empty set. Thus to make our work to be meaningful, we will construct an invariant, 
namely $e(M,\psi)$, of $(M,\psi)$ which is zero if $\psi$ is a free action, and it measures the contributions of exceptional orbits to the Chern number of the complex line orbibundle associated to $(M,\psi)$.

Now, let us define 
\[
e(M, \psi) = \langle c_1(E)^n, [B] \rangle \modz.
\]
Obviously, this number is well-defined up to $S^1$-equivariant diffeomorphism. Also, we have  $e(M, \psi) = 0$ if $\psi$ is a free action. Thus the invariant $e(M,\psi)$ is a good candidate which can be computed in terms of the local data $\mathcal{L}(M,\psi)$. 

Now, consider an $S^1$-manifold $M$ and fix a point $x$ in the interior $\mathring{M}$ of $M$. Let $C$ be an orbit of $x$ whose isotropy subgroup 
is $\Z_{p(C)}$ where $p(C)$ be the order of the isotropy subgroup of $C$. 
By the slice theorem (see Theorem \ref{theorem-slice_theorem}), 
there exists an $S^1$-equivariant neighborhood $\mathcal{U}$ of $C$ such that 
\[
\mathcal{U} \cong S^1 \times_{\Z_{p(C)}} V_x
\]
where $V_x$ is the slice representation of $\Z_{p(C)}$ at $x$. 

\begin{proposition}\label{prop-localnormalform}
Let $(M,\psi)$ be a $(2n+1)$-dimensional fixed-point-free $S^1$-manifold. 
Suppose that $C\subset\mathring{M}$ is an orbit with the isotropy subgroup 
$\Z_p$, which is possibly trivial. Then there exists an $S^1$-equivariant tubular neighborhood $\mathcal{U}$ of $C$
which is $S^1$-equivariantly diffeomorphic to $S^1 \times \C^n$ where $S^1$ acts on $S^1 \times \C^n$ by 
\[
t \cdot (w, z_1,z_2, \cdots, z_n) = (t^pw, t^{q_1}z_1, t^{q_2}z_2, \cdots, t^{q_n}z_n)
\]
for some integers $q_1, q_2, \cdots, q_n$. Moreover, the (unordered) integers {$q_j$'s} are uniquely determined 
modulo $p$.
\end{proposition}

In other words, Proposition \ref{prop-localnormalform} says that an $S^1$-equivariant tubular neighborhood of the form 
$ S^1 \times_{\Z_{p(C)}} V_x$ can be trivialized as a product space and the given action can be expressed as a linear action. 

In this paper, we deal with the case where the action is {\em pseudo-free}. Recall that an $S^1$-action on a smooth manifold $M$ is called {\em pseudo-free} if there is no fixed point and there are only finitely many exceptional orbits. Equivalently, the action on $M$ is pseudo-free if the quotient space $M / S^1$ has only isolated cyclic quotient singularities. 
Let $\mathcal{E} = \mathcal{E}(M, \psi)$ be the set of exceptional orbits of $(M,\psi)$.
Then Proposition \ref{prop-localnormalform} implies that each $C\in\mathcal{E}$ with the stabilizer $\Z_{p(C)}$ assigns a  vector 
\[
\vec{q}(C) = (q_1(C), q_2(C), \cdots, q_n(C)) \in (\Z_{p(C)}^{\times})^n,
\]
where $\Z_p^{\times}$ is a multiplicative group consisting of elements in $\Z_p$ which are coprime to $p$.
We call $\vec{q}(C)$ the {\em weight-vector}, and say that $C$ is {\em of $(p(C); \vec{q}(C))$-type}.
\begin{remark}
	Note that $\vec{q}(C)$ is unique up to ordering of $q_i(C)$'s. 
\end{remark}
Thus if the action $\psi : S^1 \times M \rightarrow M$ is pseudo-free, then the local data of $(M,\psi)$ is given by 
\[
\mathcal{L}(M,\psi) = \{ (C, (p(C); \vec{q}(C)))\}_{C\in\mathcal{E}}.
\]
In Section 4, we give an explicit formula (Theorem \ref{theorem-main})
of $e(M,\psi)$ in terms of the local data $\mathcal{L}(M,\psi)$ if $\psi$ is a pseudo-free $S^1$ action on $M$.
As a first step, we prove the following. 

\begin{proposition}\label{proposition-exactlyoneexceptionalorbittype}
Let $p > 1$ be an integer and let $\vec{q} = (q_1, \cdots, q_n) \in (\Z_p^{\times})^n$.
Then there exists a $(2n+1)$-dimensional oriented closed pseudo-free $S^1$-manifold $(M,\psi)$ having exactly one exceptional orbit $C$ of $(p;\vec{q})$-type. 
Moreover, 
\[
e(M, \psi) = \frac{q_1^{-1}q_2^{-1}\cdots q_n^{-1}}{p} \modz
\]
where $q_j^{-1}$ is the inverse of $q_j$ in $\Z_p^{\times}$.
\end{proposition}

Using Proposition \ref{proposition-exactlyoneexceptionalorbittype}, we prove our main theorem as follows. 
\begin{theorem}\label{theorem-main}
Suppose that $(M,\psi)$ is a $(2n+1)$-dimensional oriented closed pseudo-free $S^1$-manifold with the set $\mathcal{E}$ of exceptional orbits.
Then
\[
e(M,\psi) = \sum_{C\in\mathcal{E}} \frac{q_1(C)^{-1} q_2(C)^{-1}\cdots q_n(C)^{-1}}{p(C)} \modz
\]
where $q_j(C)^{-1} $ is the inverse of $q_j(C)$ in $\Z_{p(C)}^{\times}$. 
\end{theorem}

Theorem \ref{theorem-main} has particularly interesting applications when we consider a pseudo-free $S^1$-manifold $(M,\psi)$ such that 
$e(M,\psi) = 0$. In this case, our theorem gives a constraint on the local data $\mathcal{L}(M,\psi)$ given by 
\[
\sum_{C \in \mathcal{E}} \frac{q_1^{-1}(C)q_2^{-1}(C) \cdots q_n^{-1}(C)}{p(C)} \equiv 0 \modz. 
\]
As immediate applications, we can obtain the following corollaries where the proof will be given in Section~\ref{sec:applications}.

\begin{corollary}\label{corollary-thecasewheree0}
Suppose that $(M,\psi)$ is an oriented closed pseudo-free $S^1$-manifold with $e(M,\psi) = 0$. If the action is not free, then $M$ contains at least two exceptional orbits. 
If $M$ contains exactly two exceptional orbits, then they must have the same isotropic subgroup. 
\end{corollary}

\begin{corollary}\label{corollary-existencerelativelynotprime}
Suppose that $(M,\psi)$ is an oriented closed pseudo-free $S^1$-manifold with $e(M,\psi) = 0$.
If $C$ is an exceptional orbit with the isotropy subgroup $\Z_p$ for some $p >1$, there exists an exceptional orbit $C' \neq C$ with the isotropy subgroup $\Z_{p'}$ for some integer $p'$ such that $\gcd(p,p') \neq 1$.
\end{corollary}

Now, we illustrate two types of such examples. One is a product manifold equipped with a pseudo-free $S^1$-action.

\begin{proposition}\label{prop_euler_number_zero_product_case}
Let $M_1$ and $M_2$ be closed $S^1$-manifolds where $\dim M_1 + \dim M_2 = 2n+1$. 
Consider $M = M_1 \times M_2$ equipped with the diagonal $S^1$-action which we denote by $\psi$. 
If $\psi$ is fixed-point-free, then $e(M,\psi) = 0$. 
\end{proposition}

By using Theorem \ref{theorem-main} and Proposition \ref{prop_euler_number_zero_product_case}, we can prove the following. 
\begin{corollary}\label{cor_localization_isolated}
Let $(M,J)$ be a closed almost complex $S^1$-manifold. Suppose that the action preserves $J$ and that there are only isolated fixed points. 
Then, 
\[
\sum_{z \in M^{S^1}} \frac{1}{\prod_{i=1}^n q_i(z)} = 0
\] 
where $q_1(z), \cdots, q_n(z)$ are the weights of the $S^1$-representation on $T_zM$. 
\end{corollary}

\begin{remark}
Note that Corollary \ref{cor_localization_isolated} also can be obtained by the ABBV-localization theorem (see Section 5 for the detail). Thus 
the authors expect that there would be some equivariant cohomology theory which covers both the odd-dimensional theory (Theorem \ref{theorem-main})
and the even-dimensional theory (ABBV-localization theorem). This work is still in progress.
\end{remark}

The other type of examples comes from equivariant symplectic geometry as follows.
Recall that for a given symplectic $S^1$-action $\psi$ on a closed symplectic manifold $(M,\omega)$ where $[\omega] \in H^2(M;\Z)$, 
there exists an $S^1$-invariant map $\mu : M \rightarrow \R / \Z\cong S^1$ called a {\em generalized moment map} defined by 
\[
\mu(x) := \int_{\gamma_x}
i_{\underline{X}}\omega ~\modz
\]
where $x_0$ is a base point, and $\gamma_x$ is any path $\gamma_x : [0,1] \rightarrow M$ such that $\gamma_x(0) = x_0$ and $\gamma_x(1) = x$.
When $\psi$ has no fixed point, then $M$ becomes a fiber bundle over $S^1$ via $\mu$ (see \cite{CKS} for the details). 

\begin{proposition}\label{prop_euler_number_zero_symplectic_case}
	Let $(M,\omega)$ be a closed symplectic manifold equipped with a fixed-point-freee $S^1$-action $\psi$ preserving $\omega$. 
	Let $\mu : M \rightarrow \R/\Z$ be a generalized moment map and let $F_\theta = \mu^{-1}(\theta)$ for $\theta\in \R/\Z$.
	Then $e(F_\theta, \psi|_{F_\theta}) = 0$. 
\end{proposition}


Finally, here we discuss the Weinstein's theorem \cite[Theorem 1]{W} and pose some conjecture. 
\begin{theorem}\cite{W}\label{theorem_weinstein}
Let $(M,\psi)$ be a $(2n+1)$-dimensional closed oriented fixed-point-free $S^1$-manifold. Let $\alpha$ be a normalized connection 1-form on $M$. 
Then 
\[
\ell^n \cdot \int_M \alpha \wedge (d\alpha)^n \in \Z
\]
where $\ell$ is the least common multiple of the orders of the isotropy subgroups of the points in $M$.
\end{theorem}
Let $(M / S^1)_{\text{sing}}$ be the set of singular points in $M / S^1$.
Our main theorem \ref{theorem-main} implies that if $\dim (M / S^1)_{\text{sing}} = 0$, then we have 
\[
	\ell \cdot \int_M \alpha \wedge (d\alpha)^n \in \Z.
\]
We pose the following conjecture. 
\begin{conjecture}\label{conj_weinstein}
	Under the same assumption of Theorem \ref{theorem_weinstein}, we have 
	\[
		\ell^{k+1} \cdot \int_M \alpha \wedge (d\alpha)^n \in \Z.
	\]
	where $k = \dim (M / S^1)_{\text{sing}}$. 
\end{conjecture}
It is obvious that Conjecture \ref{conj_weinstein} is true when $k=0$ by Theorem \ref{theorem-main}. 
One can verify that Conjecture \ref{conj_weinstein} is true when $M$ is an odd-dimensional sphere with a fixed-point-free linear $S^1$-action 
(see Proposition \ref{prop_Chern_number_sphere}). 

This paper is organized as follows. 
In Section 2, we define a {\em local data} for a fixed-point-free $S^1$-action. In Section 3, we define a Chern class of a closed fixed-point-free $S^1$-manifold and give the explicit
computation of the Chern class of an odd-dimensional sphere equipped with a linear $S^1$-action. In Section 4, we give the complete proof of Proposition \ref{proposition-exactlyoneexceptionalorbittype} and Theorem \ref{theorem-main}. Finally in Section 5, we discuss several applications of Theorem \ref{theorem-main} and 
give the proofs of Corollary \ref{corollary-thecasewheree0}, \ref{corollary-existencerelativelynotprime}, and \ref{cor_localization_isolated}. 
Also, we deal with the examples illustrated above and give the complete proof of Proposition \ref{prop_euler_number_zero_product_case} and \ref{prop_euler_number_zero_symplectic_case}.

\section{Local invariants}

The main purpose of this section is to define a {\em local invariant} for each exceptional orbit, which is invariant under $S^1$-equivariant diffeomorphisms. To do this, we first describe a neighborhood of each orbit. 

\begin{theorem}[Slice theorem]\label{theorem-slice_theorem}\cite{Au}
Let $G$ be a compact Lie group acting on a manifold $M$. Let $x \in M$ be a point whose isotropy subgroup is $H$. 
Then there exist a $G$-equivariant tubular neighborhood $\U$ of the orbit $G\cdot m$ and a 
$G$-equivariant diffeomorphism 
\[
G \times_H V_x  \rightarrow \U
\]
where $G$ acts on $G \times_H V_x$ by 
\[
g \cdot [g',v] = [gg',v]
\]
for every $g \in G$ and $[g',v] \in G \times_H V_x$. 
Here $V_x$, called a {\em slice} at $x$, is the vector space $T_xM / T_x(G\cdot x)$ with the linear $H$-action induced by the $G$-action on $T_xM$. 
\end{theorem}

In our case, $G=S^1$ and the isotropy subgroup $H$ of $x$ is isomorphic to $\Z_p$ for some $p\ge 1$ if $x$ is not fixed by the $S^1$-action. 
The following lemma will be used frequently throughout this paper.

\begin{lemma}\label{lemma-diffeomorphismontoproductspace}
Let $m > 1$ be a positive integer and let $(w_0, w_1, \cdots, w_n)$ be the coordinate system of $S^1 \times \C^n$. 
Define an $S^1$-action on $S^1 \times \C^n$ given by 
\begin{align*}
t\cdot(w_0,w_1,\dots,w_n) &= (t^{x_0} w_0, t^{x_1}w_1,\dots, t^{x_n}w_n)
\end{align*}
for some $(x_0, x_1, \cdots, x_n) \in \Z^{n+1}$ with $\gcd(x_0, m) = 1$. 
Similarly, for $\xi = e^{\frac{2 \pi i}{m}}$, define a $\Z_m$-action on $S^1 \times \C^n$ by
\begin{align*}
\xi\cdot(w_0, w_1,\dots, w_n) &= (\xi^{m_0} w_0, \xi^{m_1}w_1,\dots, \xi^{m_n}w_n)
\end{align*}
for some $(m_0, m_1, \cdots, m_n) \in \Z^{n+1}$ with $\gcd(m,m_0) = 1$. 
Then, 
\begin{enumerate}
	\item the $S^1$-action and the $\Z_m$-action commutes, 
	\item the $\Z_m$-quotient $S^1\times_{\Z_m}\C^n$ with the induced $S^1$-action 
	is $S^1$-equivariantly diffeomorphic to $S^1\times\C^n$ with an $S^1$-action given by
	\[
	t\cdot(z_0, z_1,\dots, z_n) = (t^{x_0m}z_0,t^{-x_0a_1+x_1} z_1,\dots, t^{-x_0a_n+x_n} z_n),
	\]
	where $a_i= m_0^{-1}m_i$ modulo $m$, and 
	\item if $\Z_m$ act as a subgroup of $S^1$ on $S^1 \times \C^n$, or equivalently, if $m_i=x_i$ for every $i=0,1,\cdots, n$, then $S^1\times_{\Z_m} \C^n$ with the 	
	induced $S^1 / \Z_m$-action is equivariantly diffeomorphic to $S^1 \times \C^n$ with an $S^1$-action given by
	\[
	t\cdot(z_0,z_1,\dots, z_n) = (t^{x_0} z_0, t^{s_i}z_1,\dots,t^{s_n}z_n),
	\]
	where $s_i=m^{-1}x_i$ modulo $x_0$.
\end{enumerate}
\end{lemma}
\begin{proof}
The first claim (1) is straightforward by direct computation.
For (2), 
since $m_0$ is coprime to $m$, for each $i\ge 1$, there exist integers $a_i$ and $s_i$ such that 
\[
m_0 a_i + m s_i = m_i.
\]
Then we can easily see that $a_i = m_0^{-1}m_i$ modulo $m$.

Now, we define a map $\Phi:S^1\times_{\Z_m}\C^n\to S^1\times \C^n$ as
\[
\phi([w_0,\dots,w_n]) = (w_0^m, w_0^{-a_1}w_1, \dots, w_0^{-a_n}w_n).
\]
Then $\Phi$ is well-defined since
\begin{align*}
\Phi([\xi^{m_0}w_0, \xi^{m_1}w_1,\dots, \xi^{m_n}w_n]) &= (\xi^{m_0 m}w_0^m, \xi^{-m_0a_1+m_1}w_0^{-a_1}w_1,\dots,\xi^{-m_0a_n+m_n}w_0^{a_n}w_n)\\
&=(w_0^m, \xi^{m s_1}w_0^{-a_1}w_1,\dots, \xi^{m s_n} w_0^{-a_n}w_n)\\
&=(w_0^m, w_0^{-a_1}w_1,\dots, w_0^{-a_n}w_n) = \Phi([w_0,w_1,\dots,w_n]).
\end{align*}
The surjectivity of $\Phi$ is obvious so that it is enough to show that $\Phi$ is injective.
If 
\[
\Phi([w_0,\dots,w_n]) = \Phi([w_0',\dots, w_n']), 
\]
then
\begin{itemize}
\item $w_0^m = (w_0')^m$ and
\item $w_0^{-a_i}w_i = (w_0')^{-a_i}w_i'$ for every $i=1,2,\dots,n$.
\end{itemize}
These imply that
\begin{itemize}
\item $w_0' = \xi^{k m_0} w_0$ for some $k\in\Z$ (since $\xi^{m_0}$ is also a generator of $\Z_m$), and
\item $w_0^{-a_i}w_i = \xi^{-k m_0a_i} w_0^{-a_i} w_i'$.  
\end{itemize}
Thus we have $w_i'=\xi^{km_0a_i}w_i$ for every $i=1,2,\dots,n$.
Therefore, we have
\begin{align*}
[w_0,w_1, \dots,w_n] &= [\xi^{km_0} w_0, \xi^{k m_1} w_1, \dots, \xi^{k m_n}w_n]\\
&=[\xi^{km_0}w_0, \xi^{k (m_0a_1 +ms_1)} w_1,\dots,\xi^{k(m_0a_n+ms_n)}w_n]\\
&=[\xi^{km_0}w_0,\xi^{km_0a_1}w_1,\dots,\xi^{km_0a_n}w_n] = [w_0',w_1',\dots, w_n'].
\end{align*}
To show that $\Phi$ is $S^1$-equivariant, we define an $S^1$-action on $S^1\times\C^n$ as
\[
t\cdot(z_0, z_1,\dots, z_n) = (t^{mx_0}z_0,t^{-x_0a_1+x_1} z_1,\dots, t^{-x_0a_n+x_n} z_n).
\]
Then the $S^1$-equivariance of $\Phi$ is as following.
\begin{align*}
\Phi(t\cdot[w_0,w_1,\dots,w_n]) &= \Phi([t^{x_0}w_0,t^{x_1}w_1,\dots,t^{x_n}w_n])\\
&=(t^{mx_0}w_0^m, t^{-x_0a_1+x_1}w_0^{-a_1}w_1,\dots, t^{-x_0a_n+x_n}w_0^{-a_n}w_n)\\
&=t\cdot(w_0^m, w_0^{-a_1}w_1,\dots, w_0^{-a_n}w_n)=t\cdot\Phi([w_0,w_1,\dots, w_n]).
\end{align*}
To show (3), suppose that $\Z_m$ acts on $S^1 \times \C^n$ as a subgroup of $S^1$, i.e., $m_i=x_i$ for every $i=0,1,\cdots, n$. 
By definition of $a_i$ and $s_i$, we have
\[
-x_0a_i+x_i = -m_0a_0 + m_i = m s_i.
\]
Then $s_i = m^{-1}x_i$ modulo $x_0$ since $x_0$ is coprime to $m$. Thus for every $i=1,2,\cdots,n$, the number $-x_0a_i+x_i$ is a multiple of $m$. Hence the $S^1$-action given as above is non-effective and it has a weight-vector $(mx_0, ms_1, \cdots, ms_n)$. 
Therefore, after taking a quotient by $\Z_m$ which acts trivially on $S^1\times \C^n$, the residual $S^1 / \Z_m$-action is given as in (3). 
\end{proof}

Now, let us consider a $(2n+1)$-dimensional $S^1$-manifold $(M,\psi)$. Then for each $x\in \mathring{M}$, Theorem \ref{theorem-slice_theorem} implies that $V_x\cong\R^{2n}$ and the orbit $S^1\cdot x$ has an $S^1$-equivariant tubular neighborhood diffeomorphic to $S^1 \times_H \R^{2n}$ where $H$ 
is the isotropy subgroup of $x$. 
The following proposition states that $S^1 \times_H \R^{2n}$ is in fact $S^1$-equivariantly diffeomorphic to the product space $S^1 \times \C^n$ with a certain linear $S^1$-action. 

\begin{proposition}[Proposition \ref{prop-localnormalform}]\label{prop-localnormalformsection2}
Let $(M,\psi)$ be a $(2n+1)$-dimensional fixed-point-free $S^1$-manifold. 
Suppose that $C\subset\mathring{M}$ is an orbit with the isotropy subgroup 
$\Z_p$, which is possibly trivial. Then there exists an $S^1$-equivariant tubular neighborhood $\mathcal{U}$ 
which is $S^1$-equivariantly diffeomorphic to $S^1 \times \C^n$ where $S^1$ acts on $S^1 \times \C^n$ by 
\[
t \cdot (z_0, z_1,z_2, \cdots, z_n) = (t^pz_0, t^{q_1}z_1, t^{q_2}z_2, \cdots, t^{q_n}z_n)
\]
for some integers $q_1, q_2, \cdots, q_n$. Moreover, the (unordered) integers {$q_j$'s} are uniquely determined 
modulo $p$.
\end{proposition}

\begin{proof}
Let $x \in M$ be a point in $M$ with the isotropy subgroup $\Z_p \subset S^1$, and let $V_x\cong\R^{2n}$ be the slice at $x$.
Recall that any orientation preserving irreducible real representation of $\Z_p$ is two-dimensional, and it is isomorphic to a one-dimensional complex representation of $\Z_p$ determined by a rotation number modulo $p$. 
Thus $V_x \cong \C^n$ and a $\Z_p$-action on $S^1\times V_x$ is given by 
\[
\xi \cdot (w_0, w_1,\cdots, w_n) = (\xi w_0, \xi^{-q_1}w_1,\cdots, \xi^{-q_n}w_n)
\]
for every $(w_0,w_1,\dots,w_n)\in S^1\times V_x$
where $\xi = e^{\frac{2 \pi i}{p}}$ and $q_i$'s are integers uniquely determined modulo $p$, see \cite[p.647]{Ko}  for more details. 

Let $C$ be an orbit containing $x$. 
By the slice theorem \ref{theorem-slice_theorem}, there exists an $S^1$-equivariant tubular neighborhood $\U$ of $C$ which can be identified with $S^1\times_{\Z_p} V_x$ where the $S^1$-action on $S^1\times_{\Z_p} V_x$ is induced from the $S^1$-action on $S^1\times V_x$ given by
\[
t\cdot (w_0,w_1,\dots,w_n) = (tw_0, w_1,\dots,w_n)
\]
for every $t\in S^1$ and $(w_0,w_1,\dots, w_n) \in S^1 \times V_x$. 

Now we apply Lemma~\ref{lemma-diffeomorphismontoproductspace} with $m=p$, $x_0=m_0=1$ and $x_i=0$, $m_i=-q_i$ for $i\ge1$. Then we may choose $a_i=-q_i$ and $s_i=0$ for $i\ge 1$ so that
\[
m_0a_i+ms_i = 1\cdot (-q_i) + m\cdot 0 = -q_i = m_i.
\]
Therefore, we obtain an $S^1$-equivalent diffeomorphism
\[
\Phi:S^1\times_{\Z_p} V_x\to S^1\times \C^n,
\]
where $S^1$-action on the target is given by
\begin{align*}
t\cdot(z_0,z_1,\dots,z_n) &= (t^{mx_0}z_0,t^{-x_0a_1+x_1}z_1,\dots, t^{-x_0a_n+x_n}z_n)\\
&=(t^p z_0, t^{q_1} z_1,\dots,t^{q_n}z_n).
\end{align*}
This completes the proof.
\end{proof}

By Proposition \ref{prop-localnormalformsection2}, each exceptional orbit $C$ assigns a vector 
\[
\vec{q}(C) = (q_1(C), q_2(C), \dots, q_n(C)) \in (\Z_{p(C)})^n
\]
which is uniquely determined up to ordering of $q_i(C)$'s where $p(C)$ is an order of the isotropy subgroup of $C$. 
We call $\vec{q}(C)$ a {\em weight-vector of $C$}. 

Now, assume that $(M,\psi)$ is a $(2n+1)$-dimensional closed pseudo-free $S^1$-manifold 
and let $\mathcal{E}$ be the set of exceptional orbits. 
Then each $C \in \mathcal{E}$ is isolated so that $\gcd(p(C),q_i(C))=1$ for every $i=1,2,\cdots, n$, i.e.,
\[
\vec{q}(C) \in (\Z_{p(C)}^{\times})^n.
\]

\begin{definition}\label{def_local_invariant}
Let $(M,\psi)$ be a $(2n+1)$-dimensional pseudo-free $S^1$-manifold with the set $\mathcal{E}$ of exceptional orbits. 
\begin{enumerate}
\item A {\em local data} $\mathcal{L}(M,\psi)$ is defined by 
\[
\mathcal{L}(M,\psi) = \left\{ (C, (p(C) ;  \vec{q}(C))) ~\left| 
~p(C)\in\N, \vec{q}(C)\in \left(\Z_{p(C)}^{\times}\right)^n \right.\right\}_{C \in \mathcal{E}}.
\]
\item We call $(p(C) ; \vec{q}(C))$ the {\em local invariant} of $C$, and we say that $C$ is of {\em$(p(C);\vec{q}(C))$-type}.
\end{enumerate}
\end{definition}

\section{Chern numbers of fixed-point-free circle actions}

In this section, we give a brief review of the definition of the first Chern class of fixed-point-free $S^1$-manifolds. Also we give an explicit computation of the Chern number 
of an odd-dimensional sphere equipped with a linear action and explain how the Chern number (modulo $\Z$) can be computed in terms of a local data. 

We first review the classical result about a principal bundle over a smooth manifold.

\begin{definition}
Let $G$ be a compact Lie group and $\mathfrak{g}$ be the Lie algebra of $G$. 
Let $M$ be a principal $G$-bundle. A {\em connection form} $\alpha$ on $M$ is a smooth $\mathfrak{g}$-valued 1-form such that 
\begin{itemize}
\item $\alpha(\underline{X}) = X$ for every $X \in \mathfrak{g}$, and
\item $\alpha$ is $G$-invariant
\end{itemize}
where $\underline{X}$ is a vector field on $M$, called the {\em fundamental vector field} of $X$, defined by 
\[
\underline{X}_x := \left. \frac{d}{dt}\right|_{t=0} (\exp (tX) \cdot x)
\]
for every $x \in M$. 
\end{definition}

For a given connection form $\alpha$ on $M$, the {\em curvature form} $\Omega_{\alpha}$ associated to $\alpha$ is a $\mathfrak{g}$-valued 2-form on $M$ defined by 
\[
\Omega_{\alpha} = d\alpha + [\alpha,\alpha].
\]
In particular, if $G$ is abelian, then the Lie bracket $[\cdot, \cdot]$ vanishes so that we have $\Omega_{\alpha} = d\alpha$. 

Suppose that $G = S^1 \in \C$ be the unit circle group with the Lie algebra $\mathfrak{s}^1$.
Also, let $(M,\psi)$ be a fixed-point-free $S^1$-manifold with a connection form $\alpha$. 
Then $\alpha$ can be viewed as an $\R$-valued 1-form via a linear identification map 
$\varepsilon : \mathfrak{s}^1 \rightarrow \R$. Note that $\varepsilon$ is determined by the image $\varepsilon(X) \in \R$
where $X$ is the generator of the kernel of the exponential map $\exp \colon \mathfrak{s}^1 \rightarrow S^1$.
We say that $\alpha$ is normalized if an identification map $\varepsilon$ is chosen to be 
\[
\varepsilon(X) = 1. 
\]
Equivalently, $\alpha$ is normalized if $S^1 = \mathfrak{s}^1 / \ker (\exp) \cong \R / \Z$ and $\alpha(\underline{X}) = 1$ where $X = \frac{\partial}{\partial \theta}$ and $\theta$ is a parameter of $\R$.
In particular, if $\alpha$ is normalized, then we have 
\[
	\int_{F} \alpha = 1
\]
for any free orbit $F$ (see also Remark \ref{remark_definition_curvature}). 

The following proposition is well-known and the proof is given in \cite{Au}. But we give the complete proof here to show that it can be extended to the case of a fixed-point-free action. 

\begin{proposition}\cite{Au}\label{proposition_curvature_Chern_class}
Let $M$ be a principal $S^1$-bundle over a smooth manifold $B$ and let $\alpha \in \Omega^1(M)$ be a normalized connection 1-form on $M$. 
Then, 
\begin{itemize}
\item there exists a unique closed $2$-form $\Theta_{\alpha}$ on $B$ such that $q^* \Theta_{\alpha} = d\alpha$ where $q : M \rightarrow B$ is the quotient map, 
\item $[\Theta_{\alpha}] \in H^2(B; \R)$ is independent of the choice of $\alpha$, and 
\item $[\Theta_{\alpha}]$ is equal to the first Chern class of the associated complex line bundle $M \times_{S^1} \C$ over $B$ where $S^1$ acts on $M \times \C$ by 
\[
t\cdot (x,z) = (t\cdot x, tz)
\]
for every $t\in S^1$ and $(x,z) \in M \times \C$. 
\end{itemize}
\end{proposition}

\begin{proof}
Recall the Cartan's formula which is given by 
\[
\mathcal{L}_{\underline{X}} = i_{\underline{X}}\circ d + d \circ i_{\underline{X}}.
\]
By applying the Cartan's formula to $\alpha$, we have 
\[
\mathcal{L}_{\underline{X}}\alpha  = i_{\underline{X}}\circ d\alpha  + d \circ i_{\underline{X}} \alpha = 0.
\]
Since $i_{\underline{X}}\alpha \equiv 1$, we have $i_{\underline{X}}d\alpha = 0$, i.e. $d\alpha$ is horizontal. 
Also, by applying the Cartan's formula to $d\alpha$, we have 
\[
\mathcal{L}_{\underline{X}} d\alpha = i_{\underline{X}}d^2\alpha + d i_{\underline{X}} d\alpha = 0.
\]
Therefore, there exists a push-forward of $d\alpha$, namely $\Theta_{\alpha}$, on $B$ such that $q^* \Theta_{\alpha} = d\alpha$. It is straightforward that such a
 $\Theta_{\alpha}$ is unique. 

To prove the second statement, let $\beta$ be another connection form on $M$. Then it is obvious that $\alpha - \beta$ is $S^1$-invariant and $i_{\underline{X}} (\alpha - \beta) = 0.$ Thus there exists an 1-form $\gamma$ on $B$ such that $q^* \gamma = \alpha - \beta$.
In other words, $d\gamma = \Theta_{\alpha} - \Theta_{\beta}$ so that $[\Theta_{\alpha}] = [\Theta_{\beta}]$ 
in $H^2(B; \R)$. 
	
To prove the third statement, recall that for a given smooth manifold $N$, 
there is a one-to-one correspondence between the set of principal $S^1$-bundles over $N$ and the set of homotopy classes of maps $[N, BS^1]$ where $ES^1$ is a contractible space on which $S^1$ acts freely, and $BS^1 = ES^1 / S^1$ is the classifying space of $S^1$. 
By applying this argument to $M$, we have a map $f : B \rightarrow BS^1$ and an $S^1$-equivariant map $\widetilde{f} : M \rightarrow ES^1$ such that
\[
\xymatrix{
M \ar[r]^{\widetilde{f}} \ar[d]_{q} & ES^1 \ar[d]^{\widetilde{q}}\\
B \ar[r]^f &  BS^1
}
\]
commutes. 
	    
Now, let $\alpha_0$ be a normalized connection form on $ES^1$.  Since $\widetilde{f}$ is $S^1$-equivariant, the pull-back 
$\widetilde{f}^*\alpha_0$ is also a normalized connection form on $M$ so that we have 
$f^*\Theta_{\alpha_0} = \Theta_{\tilde f^*\alpha_0}$. 
Furthermore, the above diagram induces a bundle morphism
\[
\xymatrix{
M \times_{S^1} \C  \ar[r]^{\widetilde{f_{\C}}} \ar[d]_{q_\C} & ES^1 \times_{S^1} \C \ar[d]^{\widetilde{q}_{\C}}\\
B \ar[r]^f &  BS^1
}
\]
for any fixed linear $S^1$-action on $\C$ where $q_\C$ ($\widetilde{q}_\C$, respectively) is an extension of $q$ ($\widetilde{q}$, respectively). 
Therefore, by the naturality of characteristic classes, it is enough to show that $[\Theta_{\alpha_0}]$ is equal to the first Chern class of the complex line bundle 
\[
\mathcal{O}(1) := ES^1 \times_{S^1} \C \rightarrow BS^1
\]
where $S^1$ acts on $ES^1 \times \C$ by 
\[
t\cdot (x,z) = (t\cdot x, t z)
\]
for every $t\in S^1$ and $(x,z) \in ES^1 \times \C$. 
Then it follows from Corollary \ref{corollary_Chern_class_of_sphere}.
\end{proof}

Let us consider a fixed-point-free $S^1$-manifold $M$. Even though the action is not free, we can find a connection form as follows.

\begin{proposition}\label{prop_existence_connection_form}
Let $(M,\psi)$ be a closed fixed-point-free $S^1$-manifold. 
Then there exist an $\mathfrak{s}^1$-valued 1-form $\alpha$, called a {\em connection form} on $M$, such that 
\begin{itemize}
\item $\alpha(\underline{X}) = X$ for every $X \in \mathfrak{s}^1$, and
\item $\alpha$ is $S^1$-invariant.
\end{itemize}
\end{proposition}
\begin{proof}
Let $\ell$ be a least common multiple of the orders of the isotropy subgroups of the elements in $M$ and let $\Z_{\ell}$ be the cyclic subgroup of $S^1$ of order $\ell$. Then we have a quotient map $\pi_{\ell} : M \rightarrow M / \Z_{\ell} $ and the quotient space $M / \Z_{\ell}$ becomes an orbifold. 
Note that $S^1 / \Z_{\ell}$ acts on the quotient space $M / \Z_{\ell}$ freely so that $M / \Z_{\ell}$ is a principal 
$S^1 / \Z_{\ell}$-bundle over $B=M / S^1$. 
The slice theorem \ref{theorem-slice_theorem} implies that the quotient space $B$ is an orbifold, in particular, $B$ is paracompact, see \cite{Sa} for the detail.
Since any principal $S^1$-bundle over a paracompact space admits a connection form(c.f. \cite[Chap II]{KN}), there exists a connection form $\alpha'$ on $M / \Z_{\ell}$.
Then it is not hard to check that 
\[
\alpha = \frac{1}{\ell}\pi_{\ell}^*\alpha'
\]
is our desired 1-form. 
\end{proof}

\begin{lemma}\label{lemma_not_depending_on_the_choice_of_connection}
Let $\alpha$ be a normalized connection form on $M$.
There exists a unique closed 2-form $\Theta_{\alpha}$ on $M / S^1$ such that $q^* \Theta_{\alpha} = d\alpha$ where $q : M \rightarrow M / S^1$ is the quotient map.
Moreover, $[\Theta_{\alpha}] \in H^2(M / S^1 ;\R)$ does not depend on the choice of $\alpha$. 
\end{lemma}
\begin{proof}
The proof is exactly same as in the proof of Proposition \ref{proposition_curvature_Chern_class}.
\end{proof}

Now, we define the first Chern class of a fixed-point-free $S^1$-manifold as follows.
\begin{definition}
Let $(M, \psi)$ be a closed fixed-point-free $S^1$-manifold.
Let $\alpha$ be a normalized connection form on $M$.
Then we call $[\Theta_{\alpha}] \in H^2(M / S^1;\R)$ {\em the first Chern class (or the Euler class) of $(M,\psi)$} and we denote by $c_1(M,\psi)$. 
\end{definition}

\begin{remark}\label{remark_definition_curvature}\cite[page 194]{CdS}
The reader should keep in mind that a connection form $\alpha$ is an $\mathfrak{s}^1$-valued 1-form, and we need to identify $\mathfrak{s}^1$ with $\R$ via $\varepsilon$ to regard $\alpha$ as a usual $\R$-valued differential form. For example, Audin \cite[Example V.4.4] {Au} used an identification map $\varepsilon(\frac{\partial}{\partial\theta}) = 2\pi$ and defined the Chern class by $[\frac{1}{2\pi}\Theta_{\alpha}]$. In \cite{CdS}, Cannas da Silva used the same identification map as in our paper. 
\end{remark}

Note that since $\alpha$ is normalized, we have $\int_{S^1} \alpha = 1$. Thus if $M$ is of dimension $2n+1$, then we have 
\[
\langle c_1(M,\psi)^n, [B] \rangle = \int_B \Theta_{\alpha} \wedge \Theta_{\alpha} \wedge \cdots \wedge \Theta_{\alpha} = \int_M \alpha \wedge (d\alpha)^n
\]
where $B = M / S^1$ and $[B] \in H_{2n}(B;\Z)$ is the fundamental homology class of $B$. 

\begin{remark}
The theory of characteristic classes of orbibundles is well established in the case of {\em good orbibundles}.
In fact, they are defined as elements of {\em orbifold cohomology}. 
In our case, $B$ is the quotient space of $M$ by a pseudo-free $S^1$-action and the orbifold cohomology 
$H^*_{orb}(B)$ is the same as the equivariant cohomology $H^*_{S^1}(M)$ (see \cite[Proposition 1.51]{ALR}). 
Also, the fibration 
\[
q_{S^1} : M \times_{S^1} ES^1 \rightarrow M/ S^1 = B
\]
induces an isomorphism $q_{S^1}^* : H^*(B) \rightarrow H^*_{S^1}(M)$ with coefficients in a field (see \cite[Proposition 2.12]{ALR}).
With this identification, one can see that the Chern class $c_1(M,\psi) \in H^2(B;\R)$ defined above is actually the same as the Chern class 
$c_1(\tilde E) \in H^2(M\times_{S^1}ES^1)=H^2_{orb}(B;\R)$ (defined as in \cite[page 45]{ALR}) of the associated line bundle 
\[
\tilde E = (M \times \C)\times_{S^1} ES^1 \rightarrow M\times_{S^1} ES^1
\]
In other words, we have 
\[
q_{S^1}^* (c_1(M,\psi)) = c_1(\tilde E).
\]
\end{remark}

The following proposition gives an explicit computation of the Chern numbers of odd-dimensional spheres equipped with linear $S^1$-actions, which we use crucially to prove 
Proposition \ref{proposition-exactlyoneexceptionalorbittype} and Theorem \ref{theorem-main}.

\begin{proposition}\label{prop_Chern_number_sphere}
Suppose that an $S^1$-action $\psi$ on $S^{2n-1} \subset \C^n$ is given by
\[
t \cdot (z_1, z_2, \cdots, z_n) = (t^{p_1} z_1, t^{p_2}z_2, \cdots, t^{p_n}z_n)
\]
for some $(p_1, p_2, \cdots, p_n) \in (\Z \setminus \{0\})^n$. Then 
\[
\langle c_1(S^{2n-1},\psi)^n, [S^{2n-1} / S^1] \rangle = \frac{1}{\prod_{i=1}^n p_i}.
\]
\end{proposition}

\begin{proof}
We will use real coordinates $(x_j,y_j) = z_j = x_j + iy_j$ for $j=1,2,\cdots,n$. Recall that $\mathfrak{s^1}$ is identified with $\R$ which is parametrized by $\theta$ and $t = e^{2\pi i \theta}$.
For $X = \frac{\partial}{\partial \theta}$, we have 
\[
\underline{X} = 2\pi \sum_j p_j\left( -y_j \frac{\partial}{\partial x_j} + x_j \frac{\partial}{\partial y_j} \right).
\]
Define a connection form $\alpha$ on $S^{2n-1}$ such that 
\[
\alpha = \frac{1}{2\pi} \sum_j \frac{1}{p_j} (-y_j dx_j + x_j dy_j).
\]
Then we can easily check that $\alpha$ is a normalized connection form on $S^{2n-1}$.
By differentiating $\alpha$, we have 
\begin{align*}
d\alpha &  =  \frac{1}{2\pi} \sum_j \frac{1}{p_j}(-dy_j \wedge dx_j + dx_j \wedge dy_j)\\
&  =    \frac{1}{2\pi} \sum_j \frac{2}{p_j}dx_j \wedge dy_j\\
&  =    \frac{1}{\pi} \sum_j \frac{1}{p_j} dx_j \wedge dy_j.
\end{align*}
Therefore, we have 
\[
\int_{S^{2n-1}} \alpha \wedge (d\alpha)^{n-1} = \int_{D^{2n}} (d\alpha)^n  =  \pi^{-n} \frac{n!}{\prod_j p_j} \mathrm{Vol}(D^{2n}) = \frac{1}{\prod_j p_j}
\]
where the first equality comes from the Stoke's theorem.
\end{proof}

\begin{corollary}\label{corollary_Chern_class_of_sphere}
Let $\pi : ES^1 \rightarrow BS^1$ be the universal $S^1$-bundle and let $\alpha_0$ be a normalized connection form on $ES^1$. 
Then the curvature form $\Theta_{\alpha_0}$ on $BS^1$ represents the first Chern class of the complex line bundle 
\[
\mathcal{O}(1) = ES^1 \times_{S^1} \C
\]
where $S^1$ acts on $ES^1 \times \C$ by $ t\cdot (x,z) = (t\cdot x, tz)$ for every $t\in S^1$ and $(x,z) \in ES^1 \times \C$. 
\end{corollary}
\begin{proof}
Recall that the universal bundle $\tilde q : ES^1 \rightarrow BS^1$ can be constructed as an inductive limit of the sequence of Hopf fibrations
\[
\xymatrix@C=1.5pc{
S^3\ \ar@{^(->}[r]\ar[d]& \ S^5\ \ar@{^(->}[r]\ar[d] & &\ \cdots & S^{2n+1}\ar[d] & \cdots & \ar@{^(->}[r]&\ ES^1\ar@<-2.2pc>[d] \sim S^{\infty}\hphantom{,\C P^{\infty}} \\
\C P^1\ \ar@{^(->}[r] & \ \C P^2\ \ar@{^(->}[r] & &\ \cdots & \C P^n & \cdots & \ar@{^(->}[r] &\ BS^1 \sim \C P^{\infty},\hphantom{S^{\infty}}
}
\]
where $S^1$ acts on $S^{2n-1} \subset \C^n$ by 
\[
t \cdot (z_1,z_2,\dots, z_n) = (tz_1, tz_2, \dots, tz_n).
\]
Since $\mathcal{O}(1)$ is the dual bundle of the tautological line bundle $\mathcal{O}(-1)$ over $BS^1$, we have 
\[
c_1(\mathcal{O}(1)) = u \in H^2(BS^1;\Z)
\]
where $u$ is the positive generator of $H^2(BS^1;\Z) \cong H^2(\C P^{\infty}; \Z)$. 
Thus it is enough to show that 
\[
\left\langle [\Theta_{\alpha_0}], [\C P^1] \right\rangle =  \int_{S^3} \alpha \wedge d\alpha =  1.
\]
This follows from Proposition \ref{prop_Chern_number_sphere}.
\end{proof}

\begin{remark}\label{remark_comparison_with_Kawasaki}
In \cite[page 245]{Ka}, Kawasaki described a cohomology ring structure (over $\Z$) of the quotient space $S^{2n+1} / S^1$ where $S^1$-action $\psi$ on $S^{2n+1}$ is given by
\[
t \cdot (z_0, z_1, z_2, \cdots, z_n) = (t^{p_0} z_0, t^{p_1} z_1, \cdots, t^{p_n}z_n)
\]
for any positive integers $p_0, p_1, \dots,p_n$ such that $\gcd(p_0, p_1, \cdots, p_n) = 1$.
The ring structure of $H^*(S^{2n+1} / S^1 ;\Z)$ is as follows.
Let $\gamma_k$ be the positive generator of $H^{2k}(S^{2n+1} / S^1 ;\Z) \cong \Z$.
Then 
\[
\gamma_1 \cdot \gamma_k = \frac{\ell_1 \ell_k}{\ell_{k+1}}\gamma_{k+1}
\]
where 
\[
\ell_k = \mathrm{lcm}\left\{\left.\frac{p_{i_0}p_{i_1}\cdots p_{i_k}}{\gcd(p_{i_0}, p_{i_1}, \cdots, p_{i_k})} ~\right|~ 0\leq i_0 < \cdots < i_k \leq n\right\}.
\]
In particular, we have $\ell_1 = \operatorname{lcm}(p_0,p_1,\cdots, p_n)$ and $\ell_n = p_0p_1\cdots p_n$ since the action is effective. 
Then it is not hard to show that 
\[
\gamma_1^n = \frac{\ell_1^n}{\ell_n} \gamma_n.
\]

On the other hand, Godinho \cite[Proposition 2.15]{Go} proved that the action has the first Chern class 
\[
c_1(S^{2n+1}, \psi) = \frac{\gamma_1}{\mathrm{lcm}(p_0, p_1, \cdots, p_n)} = \frac{\gamma_1}{\ell_1}.
\]
Consequently, the Chern number is 
\[
\langle c_1(S^{2n+1}, \psi)^n, [S^{2n+1} / S^1] \rangle = \frac{1}{\ell_n} \langle \gamma_n, [S^{2n+1} / S^1] \rangle = \frac{1}{p_0p_1\cdots p_n}
\]
which coincides with Proposition \ref{prop_Chern_number_sphere}.
\end{remark}

\begin{remark}
	In \cite{Lia}, Liang studied the Chern number of a $(2n+1)$-dimensional homotopy sphere $\Sigma^{2n+1}$
	equipped with a differentiable pseudo-free $S^1$-action 
	\[
		\phi : S^1 \times \Sigma^{2n+1} \rightarrow \Sigma^{2n+1}
	\]
	under certain assumption.
	More precisely, he proved that if there are exactly $k$ exceptional orbits $C_1, \cdots, C_k$ in $\Sigma^{2n+1}$
	with isotropy subgroups $\Z_{q_1}, \cdots, \Z_{q_k}$ for some positive integers 
	$q_1, \cdots, q_k$
	such that $\gcd (q_i, q_j) = 1$ for each $i,j$ with $i \neq j$, then 
	\[
		\langle c_1(\Sigma^{2n+1}, \phi)^n, [\Sigma^{2n+1} / S^1] \rangle = \pm \frac{1}{q_1\cdots q_k}.
	\]
	His result does not involve the condition ``modulo $\Z$'' since the proof relies on the fact \cite{MY} that  
	there exists an $S^1$-equivariant map of degree $\pm 1$ from $\Sigma^{2n+1}$
	to $S^{2n+1}$ where an $S^1$-action $\phi'$ on $S^{2n+1}$ is given by 
	\[
		t \cdot (z_1, \cdots, z_{n+1}) = (t^{q_1\cdots q_k}z_1, tz_2, \cdots, tz_{n+1}).
	\]
	Thus we can obtain
	\[
		\langle c_1(\Sigma^{2n+1}, \phi)^n, [\Sigma^{2n+1} / S^1] \rangle = \pm \langle c_1(S^{2n+1}, \phi')^n, [S^{2n+1} / S^1] \rangle = \pm \frac{1}{q_1\cdots q_k}
	\]
	where the equality on the right hand side comes from Proposition \ref{prop_Chern_number_sphere}.
	Consequently, we cannot extend Liang's result (without ``module $\Z$'') 
	to a general case by the lack of such an $S^1$-equivariant map to $S^{2n+1}$.
\end{remark}

\section{Proofs of Proposition \ref{proposition-exactlyoneexceptionalorbittype} and Theorem \ref{theorem-main}}

In this section, we give the complete proofs of Proposition \ref{proposition-exactlyoneexceptionalorbittype} and Theorem \ref{theorem-main}.
Throughout this section, for a given oriented manifold $M$, we denote $M$ with the opposite orientation by $-M$. 

\begin{definition}\label{def_resolution}
Let $(M,\psi)$ be a compact oriented fixed-point-free $S^1$-manifold with free $S^1$-boundary $\partial M$, i.e., $\psi$ is free on $\partial M$.
A {\em resolution} $\mathbf{N}$ of $(M, \psi)$ is a triple $(N,\phi,h)$ consisting of a compact oriented free $S^1$-manifold $(N,\phi)$ with boundary $\partial N$
and an orientation-preserving $S^1$-equivariant diffeomorphism
$h : \partial N\rightarrow \partial M$ with respect to $\phi$ and $\psi$.
\end{definition}

\begin{remark}
Suppose that $M$ and $N$ are given as in Definition \ref{def_resolution}. 
Then $M/S^1$ has singularities, while $N / S^1$ does not. If $W$ is a singular space and if there exists a subset of $W$ which is diffeomorphic to $M / S^1$, 
then we can always remove $M / S^1$ and glue $N / S^1$ along $\partial M / S^1$. In this manner, we can think of 
\[
\widetilde{W} := \left(W \setminus (\mathring{M} / S^1)\right) \bigsqcup N /S^1
\] 		
as a resolution of $W$. This is the reason why we use the terminology `resolution' in Definition \ref{def_resolution}.
\end{remark}

Let $(M,\psi)$ be an oriented compact fixed-point-free $S^1$-manifold with free $S^1$-boundary $\partial M$, and let $\mathbf{N}=(N, \phi, h)$ be a resolution of $(M, \psi)$. 
Then we can glue $M$ and $N$ along their boundaries $\partial M$ and $\partial N$ by using $h$ as follows.

By the equivariant collar neighborhood theorem \cite[Theorem 3.5]{K}, there exist closed
 $S^1$-equivariant neighborhoods of $\partial M$ and $\partial N$ which are
 $S^1$-equivariantly diffeomorphic to $\partial M \times [0,\epsilon]$ and $\partial N \times [0,\epsilon]$, respectively,
 where $S^1$ acts on the left factors. 
Then we may extend $h$ to a map $\bar h$ on $\partial N\times[0,\epsilon]$ to $\partial M\times[0,\epsilon]$ as
\[
\begin{array}{ccccc}
\bar h & : & \partial N \times [0,\epsilon]& \rightarrow & \partial M \times [0,\epsilon]\\[0.1em]
&   &  (x,t)                                        & \mapsto     & (h(x), \epsilon - t). \\
\end{array}
\]
Note that the extended map $\bar h$ is $S^1$-equivariant and orientation-reversing. Thus we can glue $\mathring{M}$ and $-\mathring{N}$ along 
$\partial M \times (0,\epsilon)$ and $\partial N\times(0,\epsilon)$ via $\bar h$. Thus we get a closed fixed-point-free $S^1$-manifold 
\[
M_{\mathbf{N}} = \mathring{M} \sqcup _{\bar h} -\mathring{N},
\]
where the $S^1$-action $\psi_{\mathbf{N}}$ on $M_{\mathbf{N}}$ is given by  
\[
\psi_{\mathbf{N}} = \psi\sqcup_{\bar h}  \phi.
\] 
Notice that if $\psi$ is pseudo-free, then so is $\psi_{\mathbf{N}}$.

\begin{lemma}\label{lem_euler_number_indep_choice_of_resolution}
Let $(M, \psi)$ be a compact fixed-point-free $S^1$-manifold with free $S^1$-boundary $\partial M$. Suppose that there exists a resolution $\mathbf{N}$ of $(M,\psi)$.
Then $e(M_{\mathbf{N}}, \psi_{\mathbf{N}})$ is independent of the choice of a resolution $\mathbf{N}$.
\end{lemma}
\begin{proof}
Suppose that there are two resolutions $\mathbf{N}_1=(N_1, \phi_1, h_1)$ and $\mathbf{N}_2=(N_2, \phi_2, h_2)$ of $(M, \psi)$ so that we have two 
closed fixed-point-free $S^1$-manifolds 
\begin{align*}
(M_{\mathbf{N}_1}, \psi_{\mathbf{N}_1}) &= (\mathring{M}\sqcup_{\bar h_1}-\mathring{N_1}, \psi \sqcup_{\bar h_1} \phi_1), ~\text{and}\\
(M_{\mathbf{N}_2}, \psi_{\mathbf{N}_2}) &=(\mathring{M}\sqcup_{\bar h_2}-\mathring{N_2},\psi \sqcup_{\bar h_2} \phi_2).
\end{align*}
For the sake of simplicity, we denote by  
$\bar M_i = M_{\mathbf{N}_i}$ and $\bar \psi_i = \psi_{\mathbf{N}_i}$ for each $i=1,2$.
Then our aim is to prove that 
\[
e(\bar M_1, \bar\psi_1) = e(\bar M_2, \bar\psi_2).
\]

Now, let $\alpha_\partial$ be a connection form on $\partial M$. 
Then $\alpha_\partial$ can be extended to a connection form, which we still denote by $\alpha_\partial$, 
on $\partial M \times [0,\epsilon]$ via the projection $\partial M \times [0,\epsilon] \rightarrow \partial M$.
Let $\alpha$ be a connection form on $M$ such that the restriction of $\alpha$ to a closed collar neighborhood $\partial M \times [0,\epsilon]$ 
is $\alpha_\partial$. Such an $\alpha$ always exists by the existence of a partition of unity (see \cite[Theorem 2.1]{KN}).
Similarly, for each $i=1,2$, we can construct a connection form $\alpha_i$ on $N_i$ such that the restriction of $\alpha_i$ on $\partial N_i\times[0,\epsilon]$ is the pull-back $\bar h_i^*\alpha_\partial$.

Since $\alpha$ and $\alpha_i$ agree on $\partial M\times[0,\epsilon]$ and $\partial N_i\times[0,\epsilon]$ via $\bar h_i$, we can define a connection form $\bar\alpha_i$ on $\bar M_i$ by gluing $\alpha$ and $\alpha_i$ via $\bar h_i$.
Then we have 
\begin{align*}
e(\bar M_1,\bar\psi_1)-e(\bar M_2,\bar\psi_2) 
\equiv& \int_{\bar M_1}\bar\alpha_1\wedge(d\bar\alpha_1)^n - \int_{\bar M_2}\bar\alpha_2\wedge(d\bar\alpha_2)^n \\
\equiv& \int_{(\mathring{M} \setminus \partial M\times(0,\epsilon)) \sqcup -\mathring{N}_1}\bar\alpha_1\wedge(d\bar\alpha_1)^n \\
&- \int_{(\mathring{M} \setminus \partial M\times(0,\epsilon)) \sqcup -\mathring{N}_2}\bar\alpha_2\wedge(d\bar\alpha_2)^n \\
\equiv& \int_{N_2}\alpha_2\wedge(d\alpha_2)^n - \int_{N_1}\alpha_1\wedge(d\alpha_1)^n \modz.
\end{align*}

On the other hand, $h:=h_2^{-1}\circ h_1 : \partial N_1 \rightarrow \partial N_2$ is an orientation-preserving $S^1$-equivariant diffeomorphism so that 
$(N_2,\phi_2,h)$ is a resolution of $(N_1, \phi_1)$. Thus we can glue $N_1$ and $N_2$ along the collar neighborhoods of their boundaries via $\bar h$.
If we let 
\[
(\bar N,\bar \phi) = (\mathring{N_2} \sqcup_{\bar h} -\mathring{N_1}, \phi_2\sqcup_{\bar h} \phi_1),
\]
then $(\bar N, \bar\phi)$ becomes a closed free $S^1$-manifold. 
In particular, $\alpha_i$ on $\partial N_i \times [0, \epsilon]$ agree with $\bar h_i^*\alpha_\partial$ 
so that there exists a connection form $\bar\alpha$ on $\bar N$ such that 
$\bar\alpha|_{\mathring{N_i}} = \alpha_i|_{\mathring{N_i}}$ for each $i=1,2$. Consequently, since $\bar\phi$ is free, we have 
\begin{align*}
0&\equiv\int_{\bar N} \bar\alpha \wedge (d\bar\alpha)^n \modz\\
&= \int_{\mathring{N_2}\setminus\partial N_2\times(0,\epsilon)\sqcup-\mathring{N}_1} \bar\alpha \wedge (d\bar\alpha)^n\\
&=\int_{N_2} \alpha_2 \wedge (d\alpha_2)^n - \int_{N_1} \alpha_1 \wedge (d\alpha_1)^n - \int_{\partial N_2\times(0,\epsilon)} \alpha_2 \wedge (d\alpha_2)^n.
\end{align*}
Since $\alpha_2$ on $\partial N_2\times(0,\epsilon)$ is the same as $\bar h_2^*\alpha_\partial$ and $\alpha_\partial\wedge (d\alpha_\partial)^n =0$, the last term vanishes. Therefore
\[
\int_{N_2} \alpha_2 \wedge (d\alpha_2)^n - \int_{N_1} \alpha_1 \wedge (d\alpha_1)^n \equiv 0 \modz
\]
which completes the proof.
\end{proof}

In general, for a compact fixed-point-free $S^1$-manifold with free $S^1$-boundary, 
we do not know whether a resolution always exists. However, if we consider a closed tubular neighborhood of an isolated exceptional orbit, 
then a resolution always exists (see Proposition \ref{proposition-exactlyoneexceptionalorbittype}). 
To show this, suppose that there exists a closed $S^1$-manifold $(M,\psi)$ having only one exceptional orbit $C$. 
Then the local data of $(M,\psi)$ is given by 
\[
\mathcal{L}(M,\psi)=\{(C, (p; \vec q))\},
\]
for some $p=p(C) \in \N$ and $\vec q=\vec q(C) = (q_1, q_2, \dots, q_n) \in\left(\Z_p^\times\right)^n$.
Then by Proposition~\ref{prop-localnormalform}, there exists a tubular neighborhood $\mathcal{U}\cong S^1\times\C^n$ of $C$ such that the $S^1$-action is given 
by 
\[
t \cdot (w, z_1,z_2, \cdots, z_n) = (t^pw, t^{q_1}z_1, t^{q_2}z_2, \cdots, t^{q_n}z_n)
\]
for every $t \in S^1$ and $(w,z_1,z_2,\dots, z_n) \in S^1 \times \C^n$.
Observe that the complement $M \setminus \mathcal{U}$ of $\mathcal{U}$ defines a resolution of $\bar{\mathcal{U}}$.
Thus Lemma \ref{lem_euler_number_indep_choice_of_resolution} implies that $e(M,\psi)$ depends only on $(\bar{\mathcal{U}}, \psi|_{\bar{\mathcal{U}}})$, 
or equivalently, the local invariant $(p;\vec q)$ of $C$.
We first show the existence of such an $(M,\psi)$ in the case where $n = 1$.
\begin{notation}
From now on, we denote the closed unit disk in $\C$ by $D$ and identify $\bar{\mathcal{U}}$ with $S^1\times D^{n}$.
Moreover, we denote $(S^1 \times D^n, (p;q_1,q_2,\dots, q_n))$ the space $S^1 \times D^n$ equipped with an $S^1$-action given by 
\[
		t \cdot (w,z_1,z_2,\dots, z_n) = (t^pw, t^{q_1}z_1, t^{q_2}z_2, \dots, t^{q_n}z_n)
\]
for every $t\in S^1$ and $(w,z_1,z_2,\dots,z_n) \in S^1 \times D^n$.
\end{notation}
\begin{lemma}\label{lem_3dim}
Let $p > 1$ be an integer and let $q \in \Z_p^{\times}$. 
Then there exists a 3-dimensional closed pseudo-free $S^1$-manifold $(M,\psi)$ having exactly one exceptional orbit $C$ of $(p;q)$-type.
Furthermore, we have 
\[
e(M,\psi) = \frac{q^{-1}}p \modz
\]
where $q^{-1}q \equiv 1$ in $\Z_p^\times$.
\end{lemma}
\begin{proof}
By Proposition \ref{prop-localnormalform} and Definition \ref{def_local_invariant}, 
there exists an $S^1$-equivariant closed tubular neighborhood of $C$ isomorphic to $(S^1 \times D, (p;q))$.
Let $m=q^{-1}$ be the inverse of $q$ modulo $p$ and $a$ be an integer satisfying 
\[ 
pa+mq = 1.
\]
Now, let us consider a linear $S^1$-action $\bar \psi$ on $S^3 = \partial (D \times D) \subset \C^2$ given by 
\[
	t\cdot (z_1,z_2) = (t^pz_1, tz_2).
\]
We first claim that $S^3 / \Z_m$ with the induced $S^1 / \Z_m$-action, namely $\psi$, is our desired manifold $(M,\psi)$ where $\Z_m \subset S^1$ is the cyclic subgroup of $S^1$ of order $m$. 

Observe that  $S^3 =D \times S^1 \cup S^1 \times D$ so that 
\[
S^3 / \Z_m = D \times_{\Z_m} S^1 \cup S^1 \times_{\Z_m} D. 
\]
Since the $S^1$-action on $D \times S^1$ is free, the induced $S^1 / \Z_m$-action $\psi$ on $D \times_{\Z_m} S^1$ is also free.  
Thus it is enough to show that the action $\psi$ on $S^1 \times_{\Z_m} D$ has only one exceptional orbit of type $(p;q)$. 

We apply Lemma~\ref{lemma-diffeomorphismontoproductspace} with $m$, $x_0=m_0=p$, and $x_1=m_1=1$. Then we can choose $a_1=a$ and $s_1=q$
and we have a $S^1$-equivariant diffeomorphism 
\[
\Phi:S^1\times_{\Z_m} D\to S^1\times D
\]
where the target admits the residual $S^1$-action given by
\[
t\cdot(w,z) = (t^{x_0}w, t^{s_1}z) = (t^p w, t^q z).
\]
Therefore $(S^1\times_{\Z_m} D,\psi)$ is $S^1$-equivariantly diffeomorphic to $(S^1\times D, (p;q))$ and so $(S^3/\Z_m,\psi)$ has exactly one exceptional orbit of $(p;q)$-type as desired.

On the other hand, by Proposition \ref{prop_Chern_number_sphere}, we have 
\[
e(S^3, \bar \psi) = \frac{1}{p}. 
\]
Let $\underline{X}$ ($\underline{X}^m$, respectively) be the fundamental vector field on $S^3$ ($S^3 / \Z_m$, respectively) with respect to $\bar \psi$ ($\psi$, respectively).
Then the quotient map $q : S^3 \rightarrow S^3 / \Z_m$ maps the fundamental vector field $\underline{X}$ to $m\underline{X}^m$. Thus if we choose any connection form 
$\alpha$ on $S^3 / \Z_m$, then $\frac{1}{m} q^*\alpha$ is a connection form on $S^3$. Therefore, we have 
\begin{align*}
e(S^3 / \Z_m, \psi) & = \int_{S^3 / \Z_m} \alpha \wedge d\alpha = \frac{1}{m} \int_{S^3} q^*\alpha \wedge d(q^*\alpha)\\
                              & = m \int_{S^3} \frac{1}{m}q^*\alpha \wedge \frac{1}{m} d(q^*\alpha) \\
                              & = \frac{m}{p} \equiv \frac{q^{-1}}{p} \modz.
\end{align*}
\end{proof}

\begin{remark}
In Lemma \ref{lem_3dim}, $(M,\psi)$ is not unique. 
For example, if $(M,\psi)$ is given in Lemma \ref{lem_3dim} and if we perform an $S^1$-equivariant Dehn surgery along a free orbit in $(M, \psi)$, then we get a new pseudo-free $S^1$-manifold $(\tilde{M}, \tilde{\psi})$ having exactly one exceptional orbit of $(p;q)$-type. 
\end{remark}

To prove Proposition \ref{proposition-exactlyoneexceptionalorbittype}, we need the following series of lemmas.

\begin{lemma}\label{lem_extend_one_exceptional_orbit}
Suppose that $(M,\psi)$ be a $(2n-1)$-dimensional closed pseudo-free $S^1$-manifold having only one exceptional orbit $C$ of $(p;q_1,q_2,\dots,q_{n-1})$-type 
where $p\in\N$ and $(q_1,\dots,q_{n-1}) \in (\Z_p^{\times})^{n-1}$. Then there exists a $(2n+1)$-dimensional closed pseudo-free $S^1$-manifold 
$(\tilde M, \tilde \psi)$ having 
only one orbit $\tilde C$ of $(p; q_1,q_2,\cdots, q_{n-1}, 1)$-type. Moreover, we have 
$e(M,\psi) = e(\tilde M, \tilde \psi)$.
\end{lemma}
\begin{proof}
Recall that $D$ is the unit disk in $\C$. 
Let us consider a manifold $M \times D$ with an $S^1$-action $\bar \psi$ given by 
\[
t \cdot (x, z) := (t\cdot x, tz)
\]
for every $t \in S^1$ and $(x,z) \in M \times D$. 
Then it is obvious that $\bar \psi$ has only one exceptional orbit $C \times \{0\}$ of $(p; q_1,q_2,\cdots, q_{n-1}, 1)$-type. 
Thus it is enough to construct a resolution of $(M \times D, \bar \psi)$.

Let $E = M \times_{S^1} D$ with an $S^1$-action $\phi$ given by 
\[
	t \cdot [x,z] = [t\cdot x,z] = 
	[x, t^{-1} z]
\]
for every $t \in S^1$ and $[x,z] \in M \times_{S^1} D$. Then we have $\partial E = M \times_{S^1} S^1 = M$ and $\phi$ on 
$\partial E$ coincides with $\psi$. Thus the product space $N = E \times S^1$ with an $S^1$-action $\bar \phi$ given by 
\[
	t \cdot ([x,z], w) = (t\cdot [x,z], tw)
\]
has a boundary $\partial N = M \times S^1$ such that $\bar \phi|_{\partial N} = \bar \psi|_{M \times S^1}$ via the canonical identification map $h : \partial(M \times D) \rightarrow \partial N = \partial(E \times S^1)$. Obviously, $\bar \phi$ is free on $N$ so that 
$(N, \bar \phi, h)$ is a resolution of $(M \times D, \bar\psi)$ if $N$ is smooth. 
However, the problem is that $E$ is not smooth and neither is $N$ in general. In fact, there is only one isolated singularity $C \times_{S^1} \{0\}$ on the zero section $M \times_{S^1} \{0\} \subset E$ where $C$ is the unique exceptional orbit of $(M,\psi)$. Locally, a neighborhood of $C \times_{S^1} \{0\}$ is 
$S^1$-equivariantly diffeomorphic to 
\[
	(S^1 \times \C^{n-1}) \times_{S^1} \C \cong \C^{n-1} \times_{\Z_p} \C, 
\]
where $S^1$-action $\phi$ on $\C^{n-1} \times_{\Z_p} \C$ is given by 
\[
t \cdot [z_1,z_2,\cdots, z_{n-1}, z] = [z_1, z_2, \cdots, z_{n-1}, t^{-1}z]
\]
for every $t\in S^1$ and $[z_1,z_2,\cdots, z_{n-1}, z] \in \C^{n-1} \times_{\Z_p} \C$.
In other words, $C \times_{S^1} \{0\}$ corresponds to the origin $\mathbf{0}$ in $\C^{n-1}\times_{\Z_p}\C$ which is a cyclic quotient singularity fixed by $\phi$. Furthermore, it is a toroidal singularity, i.e., $\C^{n-1} \times_{\Z_p} \C$ is an affine toric variety with the isolated singularity $\mathbf{0}$ equipped with a $(\C^*)^n$-action given by 
\[
(t_1,t_2,\cdots, t_n) \cdot [z_1,z_2,\cdots,z_{n-1},z] = [t_1z_1, t_2z_2, \cdots, t_{n-1}z_{n-1}, t_nz]
\]
for every $(t_1,t_2,\cdots, t_n) \in (\C^*)^n$ and $[z_1,z_2,\cdots, z_{n-1},z] \in \C^{n-1} \times_{\Z_p} \C$
such that $\phi$ acts as a subgroup of $(\C^*)^n$. Therefore, by \cite[Theorem 11]{KKMS}, there exists a $(\C^*)^n$-equivariant resolution of $\C^{n-1} \times_{\Z_p} \C$.
Consequently, there exists a $\phi$-equivariant resolution $E'$ of $E$ with an extended $S^1$-action $\phi'$. Thus $E' \times S^1$ admits a free $S^1$-action $\bar \phi'$ given by 
\[
	t \cdot (x, w) = (t\cdot x, tw)
\]
for every $t\in S^1$ and $(x,w) \in E' \times S^1$. Since $\partial E' = \partial E$ via the canonical identification map, say $ h'$, we have a triple $(E' \times S^1, \bar\phi', h')$ which is a resolution of $(M \times D, \bar \psi)$. Therefore, we get a $(2n+1)$-dimensional closed pseudo-free $S^1$-manifold $(\tilde M,\tilde\psi)$
\[
	\tilde M = M \times D \sqcup_{\bar h'} E' \times S^1,
\]
\[
\widetilde \psi = \bar \psi \sqcup_{\bar h'} \bar \phi'
\]
where $\bar h':\partial (E' \times S^1) \times[0,\epsilon] \to \partial (M\times D)\times[0,\epsilon]$ is an $S^1$-equivariant orientation-reversing diffeomorphism defined by $h'$ as before.
Obviously, $(\widetilde M, \widetilde \psi)$ has exactly 
one exceptional orbit of type $(p;q_1,q_2,\cdots, q_{n-1}, 1)$.
 
Now, it remains to show that $e(\widetilde M, \widetilde \psi) = e(M,\psi)$. Let $\beta = d\theta$ be the normalized connection form on $D \setminus \{0\}$ with respect to 
an $S^1$-action on $D$ given by $t \cdot z = tz$ where
\[
D=\{r e^{2\pi i \theta}|r, \theta\in[0,1]\}.
\]
We consider the pull-back of $\beta|_{\partial D=S^1}$ along the natural projection $E' \times S^1 \rightarrow S^1$ and denote by $\beta$ again. Then $\beta$ becomes a normalized connection form
on $(E' \times S^1, \bar \phi')$. 

We will construct a global normalized connection form on $(\widetilde M, \widetilde \psi)$ as follows.
Let $\alpha$ be a normalized connection form on $(M, \psi)$ and let $f:[0,1]\to[0,1]$ be a smooth function such that $f(r)\equiv 0$ near $r=0$ and $f(r)\equiv 1$ near $r=1$. 
Let 
\[
\hat \alpha = (1-f(r))\alpha+f(r)\beta
\] 
be a one-form on $M \times D$ where $r=|z|$ for $z\in D$. Though $\beta$ is not defined on the whole $M \times D$, the one-form $\hat \alpha$ is well-defined on the whole $M \times D$ since $f \equiv 0$ near $r = 0$.
Moreover, it is obvious that $\hat \alpha$ is a normalized connection form on $M \times D$. 
In particular, $\hat \alpha$ coincides with $\beta$ on a neighborhood
of $\partial (E' \times S^1) = M\times S^1=\partial(M\times D)$. Thus we can glue $\hat \alpha$ and $\beta$ so that we get a global normalized connection form $\widetilde \alpha$, i.e. 
$\widetilde \alpha$ is a connection form on $(\widetilde M, \widetilde \psi)$ such that 
\[
\widetilde \alpha|_{\partial (E' \times S^1)} = \beta\quad\text{ and }\quad
\widetilde \alpha|_{M \times D} = \hat \alpha.
\]
Since $d\beta = 0$ on $E' \times S^1$, we have 
\begin{align*}
e(\widetilde M,\widetilde \psi)&\equiv \int_{\widetilde M} \widetilde\alpha \wedge (d\widetilde \alpha)^{n} \modz\\
&=\int_{M \times D} \hat\alpha \wedge (d\hat \alpha)^n
+\int_{E' \times S^1} \beta \wedge (d\beta)^n 
-\int_{\partial(E'\times S^1)\times(0,\epsilon)} \beta \wedge (d\beta)^n\\ 
&= \int_{M\times D} \hat\alpha\wedge (d\hat\alpha)^{n} + 0 + 0\\
&=\int_{M \times D} \left((1-f)^{n+1}\right)' \beta \wedge dr \wedge \alpha\wedge (d\alpha)^n\\
&=\int_0^1-\left((1-f)^{n+1}\right)' dr \int_{\partial D} \beta  \int_{M}\alpha\wedge (d\alpha)^n\\
& = e(M, \psi)
\end{align*}
which completes the proof.
\end{proof}

\begin{lemma}\label{lem_quotient_formula}
Let $(M,\psi)$ be a $(2n+1)$-dimensional 
closed pseudo-free $S^1$-manifold with exactly one exceptional orbit $C$ of $(p;\vec{q})$-type where $p\in \N$ and $\vec q \in (\Z_p^{\times})^n$.
Then for any $r \in \N$ with $\gcd(p,r) = 1$, the quotient space $M / \Z_r$ with the induced $S^1 / \Z_r$-action $\psi_r$ is also a pseudo-free $S^1$-maifold with 
exactly one exceptional orbit of type $(p; r^{-1}\vec {q})$ where $r^{-1}$ is the inverse of $r$ in $\Z_p^{\times}$. Moreover, we have 
\[
	e(M / \Z_r, \psi_r) = r^n \cdot e(M,\psi) \modz.
\]
\end{lemma}

\begin{proof}
Since $\gcd(p,r) = 1$, it is straightforward that the $\Z_r$-action on $M$ is free so that $M / \Z_r$ is a smooth manifold.
Let $\mathcal{U}$ be an $S^1$-equivariant neighborhood of $C$.
It is also obvious that $\psi$ is free on $M\setminus\mathcal{U}$ and therefore the induced $S^1 / \Z_r$-action $\psi_r$ is also free on $(M\setminus\mathcal{U})/\Z_r$. 
Therefore, there is no exceptional orbit in $M \setminus \mathcal{U}$ so that we need only to care about a neighborhood $\mathcal{U}$ of $C$.

We apply Lemma~\ref{lemma-diffeomorphismontoproductspace} with parameters $m=r$, $x_0=m_0=p$ and $x_i=m_i=q_i$ since $\Z_r$ is a subgroup of $S^1$.
Then $\mathcal{U}/\Z_r$ is $S^1$-equivalently diffeomorphic to $S^1\times\C^n$ such that the induced $S^1/\Z_r$-action on $S^1\times\C^n$ is given by
\[
t\cdot(w,z_1,\dots,z_n) = (t^p w, t^{s_1}z_1,\dots, t^{s_n} z_n) = (t^p w, t^{r^{-1}q_1}z_1,\dots,t^{r^{-1}q_n}z_n).
\]
where $r^{-1}$ is the inverse of $r$ in $\Z_p^\times$.

Now, it remains to show that 
\[
	e(M / \Z_r, \psi_r) = r^n \cdot e(M,\psi) \modz.
\]
Let $\underline{X}$ and $\underline{X}^r$ be the fundamental vector fields of $(M,\psi)$ and $(M/\Z_r,\psi_r)$, respectively. Then the quotient map $q : M \rightarrow M / \Z_r$ maps $\underline{X}$ to 
$r\underline{X}^r$. Let $\alpha$ be a normalized connection form on $M / \Z_r$. Then we can easily check that 
$\frac{1}{r} q^*\alpha$ is a normalized connection form on $M$. Then, 
\begin{align*}
	e(M/\Z_r, \psi_r) & = \int_{M / \Z_r} \alpha \wedge (d\alpha)^n = \frac{1}{r} \int_M q^*\alpha \wedge (q^*d\alpha)^n\\
	                                & = r^n \int_M \frac{1}{r}q^*\alpha \wedge (\frac{1}{r} q^*d\alpha)^n \\
	                                & = r^n \cdot e(M,\psi)
\end{align*}
which completes the proof.
\end{proof}

Now we are ready to prove Proposition~\ref{proposition-exactlyoneexceptionalorbittype}.

\begin{proposition}[Proposition \ref{proposition-exactlyoneexceptionalorbittype}]
Let $p > 1$ be an integer and let $\vec{q} = (q_1, \cdots, q_n) \in (\Z_p^{\times})^n$.
Then there exists a $(2n+1)$-dimensional oriented closed pseudo-free $S^1$-manifold $(M,\psi)$ having exactly one exceptional orbit $C$ of $(p;\vec{q})$-type. 
Moreover, 
\[
e(M, \psi) = \frac{q_1^{-1}q_2^{-1}\cdots q_n^{-1}}{p} \modz
\]
where $q_j^{-1}$ is the inverse of $q_j$ in $\Z_p^{\times}$.
\end{proposition}

\begin{proof}
Let $r_i = q_i q_{i+1}^{-1}\in\Z_p^\times$ for $i<n$ and let $r_n = q_n$. Then
\[
q_i \equiv r_ir_{i+1}\dots r_n\in\Z_p^\times
\]
for every $i=1,2,\cdots, n$. Thus $C$ is of $(p; r_1r_2\cdots r_n, \cdots, r_{n-1}r_n, r_n)$-type.

By Lemma \ref{lem_3dim}, there exists a three-dimensional closed pseudo-free $S^1$-manifold $(\widetilde M_1, \widetilde \psi_1)$ 
having exactly one orbit of type $(p; r_1)$ and 
\[
e(\widetilde M_1,\widetilde \psi_1) = \frac{r_1^{-1}}{p} = \frac{q_1^{-1}q_2}{p} \modz.
\]
By Lemma \ref{lem_extend_one_exceptional_orbit}, there exists a five-dimensional closed pseudo-free $S^1$-manifold $(M_2, \psi_2)$ 
having exactly one orbit of type $(p; r_1, 1)$ and 
\[
e(M_2,\psi_2) = e(\widetilde M_1, \widetilde \psi_1).
\]
Then, by Lemma \ref{lem_quotient_formula}, $\widetilde M_2 = M_2 / \Z_{r_2^{-1}}$ with the induced $S^1 / \Z_{r_2^{-1}}$-action $\widetilde \psi_2$
is a five-dimensional closed pseudo-free $S^1$-manifold having exactly one orbit of type $(p; r_1r_2, r_2)$ and 
\begin{align*}
	e(\widetilde M_2, \widetilde \psi_2) & = (r_2^{-1})^2 \cdot e(M_2, \psi_2) \\
	                                                            & = (q_2^{-2} q_3^2) \cdot e(\widetilde M_1, \widetilde \psi_1)\\
	                                                            & = \frac{q_1^{-1} q_2^{-1} q_3^2}{p} \modz
\end{align*}
Inductively, we get a $(2n+1)$-dimensional closed pseudo-free $S^1$-manifold $(\widetilde M_n, \widetilde \psi_n)$
having exactly one orbit of type $(p; r_1r_2\cdots r_n,\cdots, r_{n-1}r_n, r_n)$ and 
\begin{align*}
	e(\widetilde M_n, \widetilde \psi_n) & = (r_n^{-1})^{n} \cdot e(M_{n}, \psi_{n}) \\
	                                                            & = (q_n^{-n}) \cdot e(\widetilde M_{n-1}, \widetilde \psi_{n-1})\\
	                                                            & = (q_n^{-n}) \cdot \frac{q_1^{-1} q_2^{-1}\cdots  q_{n-1}^{-1}q_n^{n-1}}{p}\\
	                                                            & = \frac{q_1^{-1}q_2^{-1}\cdots q_n^{-1}}{p}
\end{align*}
which completes the proof.
\end{proof}

Now, we state and prove our main theorem \ref{theorem-main}.
\begin{theorem}[Theorem \ref{theorem-main}]
Suppose that $(M,\psi)$ is a $(2n+1)$-dimensional oriented closed pseudo-free $S^1$-manifold with the set $\mathcal{E}$ of exceptional orbits.
Then
\[
e(M,\psi) = \sum_{C\in\mathcal{E}} \frac{q_1(C)^{-1} q_2(C)^{-1}\cdots q_n(C)^{-1}}{p(C)} \modz
\]
where $q_j(C)^{-1} $ is the inverse of $q_j(C)$ in $\Z_{p(C)}^{\times}$. 
\end{theorem}

\begin{proof}
We use induction on the number of exceptional orbits. 
Suppose that $|\mathcal{E}| = 1$. Then it follows from Lemma \ref{lem_euler_number_indep_choice_of_resolution} that 
Theorem \ref{theorem-main} coincides with Proposition \ref{proposition-exactlyoneexceptionalorbittype}. 

Now, let us assume that Theorem \ref{theorem-main} holds for $|\mathcal{E}| = k-1$.
Let $(M,\psi)$ be a $(2n+1)$-dimensional 
closed pseudo-free $S^1$-manifold with $\mathcal{E}= \{C_1, C_2, \cdots, C_k\}$.

Assume that $C_1$ is of $(p;\vec q)$-type for some integers $p\ge 2$ and $\vec q\in(\Z_p^{\times})^n$.
By Proposition \ref{proposition-exactlyoneexceptionalorbittype}, there exists a $(2n+1)$-dimensional closed pseudo-free 
$S^1$-manifold $(N, \phi)$ having exactly one exceptional orbit $C'$ of $(p; \vec{q})$-type such that 
\[
e(N, \phi) = \frac{q_1^{-1}q_2^{-1}\cdots q_n^{-1}}{p} \modz.
\]
Let $\mathcal{U}$ and $\mathcal{U}'$ be $S^1$-equivariant tubular neighborhoods of $C_1$ and $C'$ respectively so that
\[
\bar{\mathcal{U}} \cong  (S^1 \times D^n, (p; \vec{q})) \cong \bar{\mathcal{U}'}
\]
where $D$ is the unit disk in $\C$. 
Let $\alpha$ be a normalized connection form on $S^1 \times D^n$ defined as a 
pull-back of a normalized connection form on $(S^1, p)$ via the projection map $(S^1 \times D^n, (p; \vec{q})) \rightarrow (S^1, p)$.
Let $\alpha_{\mathcal{U}}$ ($\alpha_{\mathcal{U'}}$, respectively) be the normalized connection form on $\mathcal{U}$ ($\mathcal{U}'$, respectively) induced by $\alpha$
via the identifications above.
Let $\alpha_M$ ($\alpha_N$, respectively) be an extension of $\alpha_{\mathcal{U}}$ ($\alpha_{\mathcal{U}'}$, respectively) to whole $M$ ($N$, respectively) so that 
\begin{itemize}
	\item $d\alpha_M = 0$ on $\mathcal{U}$ and 
	\item $d\alpha_N = 0$ on $\mathcal{U}'$.
\end{itemize}
Since there exists an obvious $S^1$-equivariant diffeomorphism $h : \partial \bar{\mathcal{U}'} \rightarrow \partial \bar{\mathcal{U}}$, we can easily check that
the triple $(N\setminus\mathcal{U}',\phi|_{N\setminus\mathcal{U}'}, h|_{\partial \bar{\mathcal{U}}'})$ is a resolution of $(M\setminus\mathcal{U},\psi|_{M\setminus\mathcal{U}})$ since $\phi$ is free on $N\setminus\mathcal{U}'$.

Now, let $(\bar M,\bar\psi)$ be a closed $S^1$-manifold obtained by gluing 
$(M\setminus\mathcal{U},\psi|_{M\setminus\mathcal{U}})$ and $(N\setminus\mathcal{U}', \phi|_{N\setminus\mathcal{U}'})$.
Since $h^* (\alpha_M|_{\partial {\mathcal{U}}}) = \alpha_N|_{\partial{\mathcal{U}'}}$, we can glue $\alpha_M|_{M \setminus \mathcal{U}}$ and 
$\alpha_N|_{N \setminus \mathcal{U}'}$ so that we get a normalized connection form $\bar{\alpha}$ on $\bar{M}$ such that 
\begin{itemize}
	\item $\bar{\alpha}|_{M \setminus \mathcal{U}} = \alpha_M$ and 
	\item $\bar{\alpha}|_{N \setminus \mathcal{U}'} = \alpha_N$.
\end{itemize}
Then,
\begin{align*}
e(M,\psi) - e(\bar M,\bar\psi) &= \int_M \alpha_M\wedge(d\alpha_M)^n - \int_{\bar M} \bar\alpha\wedge(d\bar\alpha)^n \\
&= \int_{\bar{\mathcal{U}}}  \alpha_M\wedge(d\alpha_M)^n
- \int_{-N \setminus \mathcal{U}' } \bar{\alpha_N} \wedge(d\bar{\alpha_N})^n\\
&= \int_N \alpha_N\wedge(d\alpha_N)^n
\end{align*}
where the last equality comes from the fact that 
\[
 \int_{\bar{\mathcal{U}}} \alpha_M\wedge(d\alpha_M)^n = \int_{\bar{\mathcal{U}'}} \alpha_N \wedge (d\alpha_N)^n = 0.
\]
Since $(\bar{M}, \bar{\psi})$ has $(k-1)$ exceptional orbits $C_2, \cdots C_k$, by induction hypothesis, we have 
\begin{align*}
e(M,\psi) &=e(\bar M,\bar\psi) + \int_N \alpha_N\wedge(d\alpha_N)^n \\
                &= e(\bar M,\bar\psi) + e(N,\phi) = \sum_{C\in\mathcal{E}} \frac{q_1(C)^{-1} q_2(C)^{-1}\cdots q_n(C)^{-1}}{p(C)} \modz
\end{align*}
which completes the proof.
\end{proof}

\section{Applications}\label{sec:applications}

In this section, we illustrate several applications of Theorem \ref{theorem-main}. 

Let $(M, \psi)$ be a $(2n+1)$-dimensional closed pseudo-free
$S^1$-manifold such that $e(M,\psi) = 0$. Then \ref{theorem-main} implies that 
\begin{equation}\label{equation_e=0}
	\sum_{C \in \mathcal{E}} \frac{q_1^{-1}(C)q_2^{-1}(C) \cdots q_n^{-1}(C)}{p(C)} \equiv 0 \modz
\end{equation}
where $\mathcal{E}$ is the set of exceptional orbits of $\psi$. Thus the condition $e(M,\psi)=0$ gives the constraint (\ref{equation_e=0}) on the local data
$\mathcal{L}(M,\psi)$.
We first give the proofs of Corollary \ref{corollary-thecasewheree0} and Corollary \ref{corollary-existencerelativelynotprime} as we see below.

\begin{corollary}[Corollary \ref{corollary-thecasewheree0}]
Suppose that $(M,\psi)$ is a closed oriented pseudo-free $S^1$-manifold with $e(M,\psi) = 0$. If the action is not free, then $M$ contains at least two exceptional orbits. 
If $M$ contains exactly two exceptional orbits, then they must have the same isotropic subgroup. 
\end{corollary}
\begin{proof}
Recall that the condition `pseudo-free' implies that a numerator of each summand in (\ref{equation_e=0}) 
is never zero. Thus the first claim is straightforward by Theorem \ref{theorem-main}.  If there are exactly two exceptional orbits $C_1$ and $C_2$, then 
\[
\frac{q(C_1)^{-1}}{p(C_1)} + \frac{q(C_2)^{-1}}{p(C_2)} \equiv 0 \modz
\]
where 
\[
q(C_i)^{-1} = q_1(C_i)^{-1} q_2(C_i)^{-1}\cdots q_n(C_i)^{-1} \in \Z_{p(C_i)}^{\times}
\]
 for $i=1,2$. Then $\frac{p(C_2)q(C_1)^{-1}}{p(C_1)} \equiv 0 \modz$ and 
 $\frac{p(C_1)q(C_2)^{-1}}{p(C_2)} \equiv 0 \modz$. 
Thus $p(C_1) ~|~ p(C_2)$ and $p(C_2) ~|~ p(C_1)$ so that $p(C_1) = p(C_2)$.
\end{proof}

\begin{corollary}[Corollary \ref{corollary-existencerelativelynotprime}]
Suppose that $(M,\psi)$ is an oriented closed pseudo-free $S^1$-manifold with $e(M,\psi) = 0$.
If $C$ is an exceptional orbit with the isotropy subgroup $\Z_p$ for some $p >1$, there exists an exceptional orbit $C' \neq C$ with the isotropy subgroup $\Z_{p'}$ for some integer $p'$ such that $\gcd(p,p') \neq 1$.
\end{corollary}
\begin{proof}
Let $C_1, C_2, \cdots, C_k$ be exceptional orbits. Suppose that 
\[
\gcd(p(C_1), p(C_i)) = 1 ~\text{for ~every}~ i=2,3,\cdots, k.
\]
By Theorem \ref{theorem-main}, 
\[
\frac{q(C_1)^{-1}}{p(C_1)} + \frac{K}{p(C_2)p(C_3)\cdots p(C_k)} \equiv 0 \modz
\]
for some $K \in \Z$ where $q(C_1)=q_1(C_1) q_2(C_1)\cdots q_n(C_1)$.
By multiplying both sides by $p(C_2)p(C_3)\cdots p(C_k)$, we get 
\[
\frac{q(C_1)^{-1}\cdot p(C_2)p(C_3)\cdots p(C_k)}{p(C_1)} \in \Z
\]
which is a contradiction to the fact that $\gcd(p(C_1), q(C_1)) = \gcd(p(C_1), p(C_i)) = 1$ for $i = 2,3, \cdots, k$. 
Thus there exists some $C_j \neq C_1$ with $\gcd(p(C_1), p(C_j)) \neq 1$. 
\end{proof}

Now, we illustrate two types of such manifolds. One is a product manifold as follows.

\begin{proposition}[Proposition \ref{prop_euler_number_zero_product_case}]\label{prop_section_5_1}
Let $M_1$ and $M_2$ be closed $S^1$-manifolds where $\dim M_1 + \dim M_2 = 2n+1$. 
Consider $M = M_1 \times M_2$ equipped with the diagonal $S^1$-action which we denote by $\psi$. 
If $\psi$ is fixed-point-free, then $e(M,\psi) = 0$. 
\end{proposition}
\begin{proof}
Observe that the projections $\pi_1 : M_1 \times M_2 \rightarrow M_1$ and $\pi_2 : M_1 \times M_2 \rightarrow M_2$ are $S^1$-equivariant  
and that $M_1$ or $M_2$ cannot have a fixed point. Without loss of generality, we may assume that $M_1$ does not have a fixed point. Let $\alpha_1$ be a connection form on $M_1$. Then $\alpha := \pi_1^* \alpha_1$ becomes a connection form on $M_1 \times M_2$ and it satisfies 
\[
\int_{M_1\times M_2} \alpha \wedge (d\alpha)^n = \int_{M_1} \alpha_1 \wedge (d\alpha_1)^n = 0
\] 
for a dimensional reason. 
In particular, we have 
\[
e(M_1 \times M_2, \psi) \equiv  \int_{M_1\times M_2} \alpha \wedge (d\alpha)^n = 0 ~(\mathrm{mod} ~\Z).
\]
\end{proof}

Now, we will show that Theorem \ref{theorem-main} and Proposition \ref{prop_section_5_1}
induce some well-known result on the fixed point theory of circle actions. 
Suppose that $(M,J)$ is a $2n$-dimensional closed almost complex manifold equipped with an $S^1$-action with a discrete fixed point set $M^{S^1}$. 
Then for each fixed point $z \in M^{S^1}$, there exist non-zero integers $q_1(z), q_2(z), \cdots, q_n(z)$, called {\em weights} at $z$, such that the action is 
locally expressed by 
\[
	t \cdot (z_1, z_2, \cdots, z_n) = (t^{q_1(z)}z_1, t^{q_2(z)}z_2, \cdots, t^{q_n(z)}z_n)
\]
for any $t \in S^1$ where $(z_1,z_2, \cdots, z_n)$ is a local complex coordinates centered at $z$.
Let us recall the {\em Atiyah-Bott-Berline-Vergne localization theorem} : 
\begin{theorem}\cite{AB}\cite{BV}\label{theorem_ABBV}
	For any equivariant cohomology class $\gamma \in H^*_{S^1}(M;\R)$, we have 
	\[
		\int_M \gamma = \sum_{z \in M^{S^1}} \frac{\gamma|_z}{\prod_{i=1}^n q_i(z)x}
	\]
	where $\gamma|_z \in H^*_{S^1}(z ; \R) \cong H^*(BS^1) = \R[x]$ is the restriction of $\gamma$ onto $z$. 
\end{theorem}
Note that if we apply Theorem \ref{theorem_ABBV} to $\gamma = 1 \in H^0_{S^1}(M)$, then Corollary \ref{cor_localization_isolated} is straightforward. 
However, we give another proof of Corollary \ref{cor_localization_isolated} by using Theorem \ref{theorem-main} as we see below.
\begin{corollary}[Corollary \ref{cor_localization_isolated}]
	Let $(M,J)$ be a closed almost complex $S^1$-manifold. Suppose that the action preserves $J$ and that there are only isolated fixed points. 
	Then, 
	\[
		\sum_{z \in M^{S^1}} \frac{1}{\prod_{i=1}^n q_i(z)} = 0
	\] 
	where $q_1(z), \cdots, q_n(z)$ are the weights at $z$. 
\end{corollary}

\begin{proof}
Let $p$ be an arbitrarily large prime number such that $p>q_i(z)$ for every $z \in M^{S^1}$ and $i=1,2,\cdots, n$. 
Suppose that 
\begin{equation}\label{equation_assumption}
\frac{a}{b} = \sum_{z \in M^{S^1}} \frac{1}{\prod_{i=1}^n q_i(z)} \neq 0
\end{equation}
 for some integers $a$ and $b$. Then 
\[
ab^{-1} = 	\sum_{z \in M^{S^1}} q_1(z)^{-1}q_2(z)^{-1} \cdots q_n(z)^{-1} \neq 0 ~\text{in} ~\Z_p
\]
by the assumption, where $q_i(z)^{-1}$ and $b^{-1}$ are the inverses of $q_i(z)$ and $b$ in $\Z_p^{\times}$, respectively, for every $i=1,2,\cdots, n$.

Let us consider the product space $M \times S^1$ with an $S^1$-action $\psi$ given by 
\[
t \cdot (x, w) = (t \cdot x, t^pw)
\]
for $t \in S^1$ and $(x,w) \in M \times S^1$. 
Then $( M\times S^1, \psi)$ is a pseudo-free $S^1$-manifold such that the set $\mathcal{E}$ of exceptional orbit is 
\[
	\mathcal{E} = \left\{ \{z\} \times S^1 \subset M \times S^1 ~|~ z \in M^{S^1}   \right\}.
\]
For each exceptional orbit $\{z\}\times S^1$, the local invariant is given by $(p; q_1(z), q_2(z), \cdots, q_n(z))$.
By Theorem \ref{theorem-main} and Proposition \ref{prop_section_5_1}, we have 
\[
	\sum_{z \in M^{S^1}} \frac{q_1(z)^{-1}q_2(z)^{-1} \cdots q_n(z)^{-1}}{p} \equiv 0 \modz.
\]
which is equivalent to 
\[
	\sum_{z \in M^{S^1}} q_1(z)^{-1}q_2(z)^{-1} \cdots q_n(z)^{-1} = 0 ~\text{in} ~\Z_p.
\]
which leads a contradiction.
\end{proof}

The other type of manifolds having $e = 0$ comes from equivariant symplectic geometry. Here we give a brief introduction to the theory of circle actions on 
symplectic manifolds. 

Let $M$ be a $2n$-dimensional closed manifold. 
A differential 2-form $\omega$ on $M$ is called a {\em symplectic form} if $\omega$ is closed and
non-degenerate, i.e.,
\begin{itemize}
    \item $d\omega = 0$, and
    \item $\omega^n$ is nowhere vanishing.
\end{itemize}
We call such a pair $(M,\omega)$ a {\em symplectic manifold}.
A smooth $S^1$-action on $(M,\omega)$ is called {\em symplectic} if it preserves $\omega$. 
Equivalently, an $S^1$-action is symplectic if $\mathcal{L}_{\underline{X}} \omega = di_{\underline{X}}\omega = 0$ where $\underline{X}$ is the fundamental vector field on $M$ generated by the action.
Thus if the action is symplectic, then $i_{\underline{X}}\omega$ is a closed 1-form so that it represents some cohomology class $[i_{\underline{X}}\omega] \in H^1(M;\R)$.
Now, let us assume that $\omega$ is integral so that $[\omega] \in H^2(M;\Z)$. By a direct computation, we can easily check that $i_{\underline{X}}\omega$ is also integral. Thus we can
define a smooth map $\mu : M \rightarrow \R / \Z\cong S^1$ such that
\[
\mu(x) := \int_{\gamma_x}
i_{\underline{X}}\omega ~\mod \Z
\]
where $x_0$ is a base point and $\gamma_x$ is any path $\gamma_x : [0,1] \rightarrow M$ such that $\gamma_x(0) = x_0$ and $\gamma_x(1) = x$.
We call $\mu$ a {\em generalized moment map}.

\begin{lemma}\cite{McD}\cite[Proposition 2.2]{CKS}\label{lemma_dmu}
	Let $\mu$ be a generalized moment map. Then 
	\[
	d\mu = i_{\underline{X}}\omega.
	\]
\end{lemma}
It is immediate consequences of Lemma \ref{lemma_dmu} that $\mu$ is $S^1$-invariant and the set of critical points of $\mu$ is equal to $M^{S^1}$.
Let $\theta \in \R / \Z$ be a regular value of $\mu$. Then $F_\theta := \mu^{-1}(\theta)$ is a $(2n-1)$-dimensional closed fixed-point-free $S^1$-manifold. 
Note that the restriction $\omega|_{F_\theta}$ has one-dimensional kernel generated by $\underline{X}$ on $F_\theta$. Thus $\omega|_{F_\theta}$ induces a symplectic structure
$\omega_\theta$ on the quotient $M_\theta := F_\theta / S^1$ and we call $(M_\theta, \omega_\theta)$ the {\em symplectic reduction at $\theta$}. 
If we choose $\epsilon > 0$ small enough so that $I_\theta := (\theta - \epsilon, \theta + \epsilon) \subset \R / \Z$ has no critical value, then 
\[
\mu^{-1}(I_\theta) \cong M_\theta \times I_\theta. 
\]
Thus we can compare $[\omega_{\vartheta}]$ with $[\omega_\theta]$ in $H^2(M;\R)$ whenever $\vartheta \in I_\theta$. 
The following theorem due to Duistermaat and Heckman gives an explicit variation formula of reduced symplectic forms. 
\begin{theorem}\cite{DH}\label{theorem_DH}
	Let $\psi_\theta$ be the induced $S^1$-action on $F_\theta$. Then 
	\[
	[\omega_{\vartheta}] - [\omega_\theta] = (\vartheta-\theta) \cdot c_1(F_\theta, \psi_\theta)
	\]
	for every $\vartheta \in I_\theta$.
\end{theorem}
Now, we can define a function, called the {\em Duistermaat-Heckman function}, on $I_\theta$ such that 
\begin{align*}
	\text{DH} : & ~I_\theta \rightarrow \R \\
	         &  ~\vartheta   \mapsto \text{Vol}(M_{\vartheta}, \omega_{\vartheta})
\end{align*}
where $\text{Vol}(M_{\vartheta}, \omega_{\vartheta})$ is a symplectic volume given by 
\[
	\text{Vol}(M_\vartheta, \omega_\vartheta) = \int_{M_\vartheta} \omega_\vartheta^{n-1}.
\]
By Theorem \ref{theorem_DH}, the Duistermaat-Heckman function $\text{DH}(\vartheta)$ is a locally polynomial of degree $n-1$ with the leading coefficient 
$\langle c_1(F_\theta, \psi_\theta)^{n-1}, [M_\theta] \rangle$. In other words, 
\begin{align*}
	\text{DH}(\vartheta) & = \left( \int_{M_\theta} c_1(F_\theta, \psi_\theta)^{n-1} \right) (\vartheta-\theta)^{n-1} + \cdots +  \int_{M_\theta} \omega_\theta^{n-1}\\
	                      & = \left( \int_{M_\theta} c_1(F_\theta, \psi_\theta)^{n-1} \right) \vartheta^{n-1} + \cdots.
\end{align*}

\begin{proposition}[Proposition~\ref{prop_euler_number_zero_symplectic_case}]
	Let $(M,\omega)$ be a closed symplectic manifold equipped with a fixed-point-free $S^1$-action $\psi$ preserving $\omega$. 
	Let $\mu : M \rightarrow \R/\Z$ be a generalized moment map and let $F_\theta = \mu^{-1}(\theta)$ for $\theta\in \R/\Z$.
	Then $e(F_\theta, \psi|_{F_\theta}) = 0$. 
\end{proposition}
\begin{proof}
Since the $S^1$-action $\psi$ on $(M,\omega)$ is fixed-point-free by assumption, the Duistermaat-Heckman function $\text{DH}$ is a polynomial defined on the whole $\R / \Z$. 
Since any periodic polynomial is a constant function, all coefficients of $\vartheta^{n-i}$ in $\text{DH}(\vartheta)$ are zero for $1\le i<n$. Indeed, the coefficient of $\vartheta^{n-i}$ can be expressed as 
\[
\sum_{j=1}^i \binom{n-j}{i-j}(-\theta)^{i-j}\int_{M_\theta} c_1(F_\theta, \psi_\theta)^{n-j}\omega^{j-1}
\]
for every $i=1,\dots, n$. In particular, we have $e(F_\theta, \psi_\theta) = 0$ when $i=1$. 
\end{proof}

Furthermore, we have the following corollary.

\begin{corollary}\label{corollary_DHfunction}
	Let $(M,\omega)$ be a $(2n+2)$-dimensional closed symplectic manifold with a fixed-point-free symplectic $S^1$-action $\psi$. 
	Assume that $[\omega] \in H^2(M; \Z)$ and every submanifold fixed by some non-trivial finite subgroup of $S^1$ is of dimension two. 
	Then we have 
	\[
	\sum_{S \in \mathcal{J}} \omega(S) \cdot \frac{q_1^{-1}(S)q_2^{-1}(S) \cdots q_n^{-1}(S)}{p(S)} \equiv 0 \modz
	\]
	where $\mathcal{J}$ is the set of connected submanifolds of $M$ having non-trivial isotropy subgroups, 
	$\omega(S)$ is the symplectic area of $S$, $p(S)$ is the order of the isotropy subgroup of $S$, 
	$(q_1(S),\cdots,q_n(S))$ is the weight-vector of $\Z_{p(S)}$-representation on the normal bundle over $S$, and $q_i(S)^{-1}$ is the inverse of $q_i(S)$ in $\Z_{p(S)}^{\times}$ for every $i=1,\cdots, n$.
\end{corollary}

\begin{proof}
	Let $\mu : M \rightarrow \R / \Z$ be a generalized moment map for $\psi$. 
	Without loss of generality, by scailing $\omega$ if necessary, 
	we may assume that $\mu^* dt = i_{\underline{X}}\omega = d\mu$ where $dt$ is a volume form on $\R / \Z$ such that 
	$\int_{\R / \Z} dt = 1$ and $\underline{X}$ is the vector field generated by the $S^1$-action $\psi$, see \cite[p. 273]{Au} for more details.
	Since the action is fixed-point-free, there is no critical point of $\mu$. 
	Let $\theta \in \R / \Z$ and we denote $\psi_\theta$ the induced action on $F_\theta = \mu^{-1}(\theta)$. 
	
	Let $\mathcal{J} = \{S_1, \cdots, S_k\}$ be the set of connected symplectic submanifolds of $(M,\omega)$ having non-trivial isotropy subgroups.
	Since each $S_i$ is two-dimensional and the induced action on $(S_i, \omega|_{S_i})$ is fixed-point-free and symplectic, 
	we can easily see that $S_i$ is diffeomorphic to 
	$T^2$ and the restriction $\mu|_{S_i}$ becomes a generalized moment map for the induced symplectic $S^1$-action on $(S_i, \omega|_{S_i})$. 
	Furthermore, each level set $(\mu|_{S_i})^{-1}(t)$ is the union of finite number of $S^1$-orbits for every $t \in \R / \Z$.
	
	Thus $F_\theta \cap S_i = (\mu|_{S_i})^{-1}(\theta)$ 
	is the union of finite number of $S^1$-orbits for each $i=1,\cdots,k$. We denote the number of connected components 
	of $F_\theta \cap S_i$ by $n_i$. 
	Consequentely, there are exactly $n_1 + \cdots +n_k$ exceptional $S^1$-orbits in $F_\theta$ and hence 
	$(F_\theta, \psi_\theta)$ is a pseudo-free $S^1$-manifold. 
	By Theorem \ref{theorem-main}, we have 
	\[
	\sum_{i=1}^k n_i \cdot \frac{q_1^{-1}(S_i)q_2^{-1}(S_i) \cdots q_n^{-1}(S_i)}{p(S_i)} \equiv 0 \modz.
	\]
	Observe that $n_i = \omega(S_i)$ since if we choose a loop $\gamma_i : S^1 \rightarrow S_i \cong T^2$
	generating a gradient-like vector field with respect to $\mu|_{S_i}$, then 
	\[
		\int_{S_i} \omega = \int_{\gamma_i} i_{\underline{X}}\omega = \langle dt, \mu_*[\gamma] \rangle = n_i
	\]
	for every $1 \leq i \leq k$. 
	This completes the proof.
\end{proof}

\begin{remark}
	Any effective fixed-point-free symplectic circle action on a closed symplectic four manifold satisfies the condition in Corollary \ref{corollary_DHfunction}.
\end{remark}

\begin{acknowledgements}
The authors thank anonymous referee for helpful suggestions and comments.
This work was supported by IBS-R003-D1.
\end{acknowledgements}

\end{document}